\documentclass[11pt,twoside]{article}
\usepackage[textwidth=6in,textheight=8.2in,centering]{geometry}
\usepackage{fancyhdr}
\usepackage{amsmath}
\usepackage{amssymb}
\usepackage{amsthm}
\usepackage{setspace}
\usepackage{mathrsfs}
\usepackage{graphicx}
\usepackage{accents}
\usepackage{hyperref}
\usepackage[english]{babel}
\usepackage{enumitem}
\usepackage{float}
\usepackage{comment}
\usepackage{hyperref}
\hypersetup{
    colorlinks=true,
    linkcolor=blue}

\DeclareMathOperator*{\esssup}{ess\,sup}

\DeclareMathOperator*{\essinf}{ess\,inf}
\DeclareMathOperator*{\leb}{Leb}

\theoremstyle{definition}
\newtheorem{theorem}{Theorem}[section] 
\theoremstyle{definition}
\newtheorem{lemma}[theorem]{Lemma}

\theoremstyle{definition}

\theoremstyle{definition}
\newtheorem{corollary}[theorem]{Corollary}
\theoremstyle{definition}
\newtheorem{remark}[theorem]{Remark}
\theoremstyle{definition}
\newtheorem{definition}[theorem]{Definition}
\newtheorem{proposition}[theorem]{Proposition}

\newtheorem{example}[theorem]{Example}
\theoremstyle{definition}

\usepackage{amsmath, amssymb}
\usepackage[dvipsnames]{xcolor}
\usepackage{tikz}

\usetikzlibrary{arrows.meta}
\usetikzlibrary{positioning}

\newcommand {\Sec}[1] {Section~\ref{#1}}

\newcommand {\fig}[1] {Figure~\ref{#1}}

\newcommand {\thrm}[1] {Theorem~\ref{#1}} 
\newcommand {\cor}[1] {Corollary~\ref{#1}} 
\newcommand {\prop}[1] {Proposition~\ref{#1}} 
\newcommand {\dfnn}[1] {Definition~\ref{#1}} 
\newcommand {\rem}[1] {Remark~\ref{#1}} 
\newcommand {\lem}[1] {Lemma~\ref{#1}}

\begin{document}

\title{\textbf{PREVALENCE OF STABILITY FOR SMOOTH BLASCHKE PRODUCT COCYCLES FIXING THE ORIGIN}}
\date{}

\author{Cecilia Gonz\'alez-Tokman\thanks{School of Mathematics and Physics, University of Queensland, St Lucia QLD 4072, Australia. \\
\texttt{cecilia.gt@uq.edu.au}}, Joshua Peters\thanks{School of Mathematics and Physics, University of Queensland, St Lucia QLD 4072, Australia. \\
\texttt{joshua.peters@uq.net.au}}}

\maketitle

\begin{abstract}\noindent
 This work investigates the stability properties of Lyapunov exponents of transfer operator cocycles from a measure-theoretic perspective. Our results focus on so-called Blaschke product cocycles, a class of random dynamical systems amenable to rigorous analysis. We show that prevalence of stability is related to the dimension of the base system's domain, $\Omega$. When $\Omega=S^1$, we show that stability is prevalent among smooth monic quadratic Blaschke product cocycles fixing the origin by constructing a so-called probe. For higher dimensional $\Omega$, we show that a probe does not exist, thus providing strong evidence that stability is not prevalent in this setting. Finally, through a perturbative method we show that almost every smooth Blaschke product cocycle fixing the origin is stable.    
\end{abstract}

\newpage
\begingroup
\hypersetup{linkcolor=[rgb]{0.0,0.0,0.0}}
\tableofcontents
\endgroup
\newpage
\noindent
\section{Introduction}
It is uncommon to be able to encapsulate all information in dynamical systems which exactly model reality. In particular, global properties of externally forced systems may alter under perturbations of system parameters -- begging the question of \textit{how and when do small errors in dynamical models drastically influence relevant outputs?} Answering this question is especially important in applications where one wishes to understand whether errors in input data can impact long term predictions.
A popular and powerful tool for the study of chaotic systems is the (Ruelle--)Perron--Frobenius operator, or transfer operator. We refer the reader to \cite{Baladi} for a monograph on the topic. This operator describes how distributions or ensembles of trajectories evolve under the dynamics of a system, providing one with information about invariant measures, rates of decay of correlations, and many statistical properties such as central limit theorems. It is natural to investigate the stability properties of this operator to determine whether or not drastic qualitative changes can occur under small perturbations. \\ \\ \noindent Significant progress has been made in the study of perturbations for Perron-Frobenius operators.
In the case of autonomous dynamical systems,
Keller and Liverani revealed in \cite{Keller} that transfer operators satisfying a Lasota-Yorke type inequality are stable under a broad class of perturbations. Specifically, they illustrated that both the isolated eigenvalues and corresponding eigenspaces continuously depend on the perturbation. Froyland, Gonz\'alez-Tokman and Quas \cite{Hilbert,LasotaYorke,gonzaleztokman2018stability}; Dragi$\check{\mathrm{c}}$evi\'c, Rugh and Sedro \cite{HHJ_Reglin,DS_spechyp}; and Crimmins \cite{crimmins}, have been developing extensions of this result in a non-autonomous setting, where instead of the Perron-Frobenius operator of a single map, a cocycle is formed, where the fibre maps are Perron-Frobenius operators associated to the maps being composed.
Instead of studying eigenvalues of operators, Lyapunov exponents of operator cocycles are investigated. In related settings,  Ma\~n\'e \cite{Mane}; Bochi and Viana \cite{instab}; and Bessa and Carvalho \cite{BC_gen}, investigate generic properties of Lyapunov exponents in the $C^0$ and $C^1$ topologies. In particular, for maps, \cite{instab} demonstrates that there exists a residual subset of volume-preserving $C^1$ diffeomorphisms such that, for almost every orbit, either all Lyapunov exponents are equal to zero, or the Oseledets splitting is dominated. Furthermore, for compact semiflows on Hilbert spaces, \cite{BC_gen} proves that there is $C^0$-residual subset of cocycles within which, for almost every orbit, either the Oseledets splitting is dominated, or all the Lyapunov exponents are equal to $-\infty$. \\ \\ \noindent In \cite{LasotaYorke}, Froyland, Gonz\'alez-Tokman and Quas established general conditions for which
(random) absolutely continuous invariant measures of interval maps depend continuously under
various perturbations of the operator cocycle. The first result answering the question of stability of the full Lyapunov spectrum and Oseledets spaces in infinite dimensions was presented by Froyland, Gonz\'alez-Tokman and Quas in \cite{Hilbert}, by considering Hilbert-Schmidt cocycles on a separable Hilbert space with exponentially decaying entries. Crimmins provides conditions in \cite{crimmins}, which guarantee stability of the Lyapunov spectrum and Oseledets splitting for the Perron-Frobenius operator cocycle acting on $C^k$ expanding maps of $S^1$. Subject to uniformly small fibre-wise $C^{k-1}$ perturbations to the random dynamics, and perturbations generated by numerical approximations, stability conditions are provided.\\ \\ \noindent In \cite{gonzaleztokman2018stability}, Gonz\'alez-Tokman and Quas gave results in the random infinite dimensional setting where they considered a  cocycle of Perron-Frobenius operators acting on a Banach space of analytic function on an annulus. They illustrated that small natural perturbations may cause a complete collapse of the Lyapunov spectrum. In the context of invertible matrix cocycles, collapse of Lyapunov exponents had been established in \cite{instab}, using alternative techniques. In \cite{gonzaleztokman2018stability}, the authors provided necessary and sufficient conditions for the Lyapunov spectrum of the Perron-Frobenius operator cocycle to exhibit stability (or instability) properties to perturbations. They further established that so-called expanding stable Blaschke product cocycles are topologically generic. Such systems were considered by \cite{gonzaleztokman2018stability} to both generalise autonomous results of expanding circle maps in \cite{MR3592677}, and to investigate the robustness of the Lyapunov--Oseledets spectrum of Perron-Frobenius operator cocycles under perturbations. We refer the reader to e.g. \cite{addit,environment,semiinvert} and references therein for further information in this direction. \\ \\ \noindent This paper aims to extend on results from \cite{gonzaleztokman2018stability}, by investigating whether
 stability of Lyapunov exponents for transfer operator cocycles (related to stability aspects of random dynamical systems from a quenched perspective) is a measure-theoretically generic property among
expanding Blaschke product cocycles. For this, we consider the prevalence framework: a translation invariant ``almost every'' condition on infinite dimensional spaces \cite{HuntPrevalence}.
See also \cite{MR1191479} for a discussion on related notions. As opposed to defining a measure on the space considered, prevalence is defined in terms of a class of compactly supported probability measures \cite{HuntPrevalence}. Applications of this theory to fields include functional analysis \cite{functional_prevalence}, smooth dynamical systems \cite{smooth_dynamical}, and geometry \cite{geometry_prevalence}. Due to the robust checkable conditions required to understand the stability of the Lyapunov spectrum for the operator cocycle considered in \cite{gonzaleztokman2018stability}, it serves as a prime candidate to investigate measure-theoretic genericity. \\ \\ \noindent In Section \ref{sec:prelim}, we introduce relevant definitions and results which will be used predominately throughout this paper. Here we formally describe the prevalence framework and present preliminary definitions and results related to random dynamical systems and transfer operators. Section \ref{sec:fbc} focuses on introducing finite Blaschke products, where we record important results and relevant spaces used throughout the paper. In Section \ref{sec:fbc0} we develop and specialise results for Blaschke product cocycles fixing the origin. {As is mentioned in Section \ref{sec:fbc0}, this is motivated by the fact that every Blaschke product cocycle may be \textit{conjugated} to another cocycle that fixes the origin.} Results from Section \ref{sec:fbc0} are utilised in Section \ref{sec:monicquad} to demonstrate that, {over the circle}, stable monic quadratic Blaschke product cocycles fixing the origin are prevalent. {We then provide strong evidence in Section \ref{sec:nonprev} that by increasing the dimension of the driving system, prevalence is not preserved.} To conclude, in Section~\ref{sec:aefbcstab} we introduce a one-parameter family of perturbations and prove that for almost every perturbation within this family, the resulting Blaschke product cocycle is stable.

\section{Preliminaries}
\label{sec:prelim}
In this section, we collate definitions and results relevant to this paper. This involves introducing the prevalence framework, and discussing random dynamical systems.
\subsection{Prevalence framework}
The theory of prevalence provides a translation invariant `almost everywhere' condition on infinite dimensional vector spaces. The following definitions were sourced from the foundational paper \cite{HuntPrevalence} and the survey \cite{Prevalence}.\\ \\ \noindent The `almost every' condition on infinite dimensional spaces was introduced by \cite{HuntPrevalence} through the notion of a \textit{prevalent set}.
\begin{definition}
        Let $E$ be a completely metrisable topological vector space. A Borel set $U\subset E$ is said to be \textit{prevalent} if there exists a Borel measure $\mu$ on $E$ such that
        \begin{itemize}
            \item[(a)] $0<\mu(V)<\infty$ for some compact subset $V$ of $E$;
            \item[(b)] The set $U+x$ has full $\mu$-measure for all $x\in E$.
        \end{itemize}
        \label{def:Prevalent}
        \end{definition}
        \noindent
In a more general setting, we say that a subset $S\subset{E}$ is prevalent if $S$ contains a prevalent Borel set \cite{Prevalence}.
\begin{definition}
A finite dimensional subspace $P\subset E$ is said to be a \textit{probe} for a set $S\subset E$ if there exists a Borel set $U\subset S$ such that $U+x$ has full Lebesgue measure on $P$ for all $x\in E$.
\label{def:probe}
\end{definition}
\noindent
The existence of a probe is a sufficient condition for a set $S$ to be prevalent. If one finds that $S\subset E$ is prevalent, we say that almost every element of $E$ lies in $S$. One can think of the measure $\mu$ as a means of describing a set of perturbations in $E$ \cite{Prevalence}. In \cite{Prevalence,HuntPrevalence,Huntctsnowherediff}, {one can find examples where the prevalence framework is applied in various settings. An example in the context of dynamical systems is the following, taken from \cite[Proposition 7]{HuntPrevalence}.}
\begin{itemize}
    \item  For $1\leq k \leq \infty$, a $C^k(\mathbb{R}^n,\mathbb{R}^n)$ map having the property that all of its fixed points are hyperbolic (and further, that its periodic points of all periods are hyperbolic) is prevalent.
\end{itemize}
In this example, the authors use a finite dimensional subspace of polynomial functions as a probe.

\subsection{Random dynamical systems and Perron-Frobenius operators}
We will be considering \textit{semi-invertible random dynamical systems}. These can be used to model a wide range of phenomena including systems with changing environments. Theoretical aspects of this field have been developed extensively in \cite{Kifer,ArnoldRDS} where the following definitions have been sourced.
\begin{definition}
\sloppy A \textit{semi-invertible random dynamical system} is a tuple $(\Omega,\mathfrak{F},\mathbb{P},\sigma,X,\mathcal{L})$, where the base $\sigma:\Omega\to \Omega$ is an \textit{invertible}\footnote{$\sigma^{-1}$ is measurable and exists for $\mathbb{P}$-a.e. $\omega\in \Omega$} measure-preserving transformation of the probability space $(\Omega,\mathfrak{F},\mathbb{P})$, $(X,\| \cdot \|)$ is a Banach space, and $\mathcal{L}:\Omega\to \mathfrak{L}(X,X)$ is a family of bounded linear operators of $X$, called the generator.
\label{def:seminvertrds}
\end{definition}
\begin{remark}
    In general, $X$ can be any given Banach space. In this paper, we will be interested in $X$ being the Banach space of analytic functions on an annulus. Generators acting on this space were studied in \cite{MR3592677}.
\end{remark}
\begin{definition}
    A measurable function $f:S\to S$ on a measure space $(S,\mathfrak{D},\mu)$ is a \textit{non-singular transformation} if $\mu(f^{-1}(D))=0$ for all $D\in \mathfrak{D}$ such that $\mu(D)=0$.
\end{definition}\noindent
For a {non-singular transformation} $f: S\to S$ of a measure space $(S,\mathfrak{D},\mu)$, we can associate a corresponding \textit{Perron-Frobenius operator}.  Since {it describes} the evolution of ensembles of points, or densities, such operators serve as a powerful tool in studying the statistical behaviour of trajectories of $f$.
\begin{definition}
\label{def:pfoperator}
    Let $(S,\mathfrak{D},\mu)$ be a measure space and $f:S\to S$ be a non-singular transformation.
    The unique operator $\mathcal{L}_f:L^1(\mu)\to L^1(\mu)$ satisfying the dual relation
    $$\int_D \mathcal{L}_{f} g(x)\, d{\mu}(s) = \int_{f^{-1}(D)}g(x)\, d{\mu}(s),$$
    for every $D\in\mathfrak{D}$ and $g\in L^1(\mu)$ is called the \textit{Perron-Frobenius operator} associated with $f$.
\end{definition}\noindent
In certain cases, $\mathcal{L}_f$ may be restricted (or extended) to a bounded linear operator on another Banach space $X$ in which case the operators are still referred to as Perron-Frobenius operators. Combining the notions from \dfnn{def:seminvertrds} and \dfnn{def:pfoperator}, one can form a semi-invertible random dynamical system from a family of Perron-Frobenius operators $(\mathcal{L}_{f_\omega})_{\omega\in\Omega}$ associated to a set of non-singular transformations $(f_\omega)_{\omega\in\Omega}.$ This forms a Perron-Frobenius operator cocycle.
\begin{example}[Perron-Frobenius operator cocycle]
Consider a semi-invertible random dynamical system $(\Omega,\mathfrak{F},\mathbb{P},\sigma,X,\mathcal{L})$ where $\sigma:\Omega\to \Omega$ is an invertible ergodic measure-preserving transformation, and its generators $\mathcal{L}:\Omega \to \mathfrak{L}(X,X)$ are the Perron-Frobenius operators associated to the non-singular transformations $f_\omega$ of the measure space $(S,\mathfrak{D},\mu)$ given by $\omega \mapsto \mathcal{L}_{f_\omega}$. This gives rise to a Perron-Frobenius operator cocycle
$$(n,\omega)\mapsto \mathcal{L}_{f_\omega}^{(n)}=\mathcal{L}_{f_{\sigma^{n-1}\omega}}\circ \cdots \circ \mathcal{L}_{f_\omega}$$
where $\mathcal{L}_{f_\omega}:X\to X$ is the Perron-Frobenius operator of $f_\omega$. Here the evolution of a density is governed by a cocycle of Perron-Frobenius operators driven by the base dynamics $\sigma:\Omega\to \Omega$.
\end{example}
\begin{definition}
The cocycle $(\mathcal{L}_\omega)_{\omega\in\Omega}$ on a Banach space $X$ is \textit{strongly measurable} if for any fixed $f\in X$, $\omega\mapsto\mathcal{L}_\omega f$ is $(\mathfrak{F}_\Omega,\mathfrak{F}_X)$-measurable.
\end{definition}\noindent
In this definition, $\mathfrak{F}_\Omega$ denotes the $\sigma$-algebra over $\Omega$, and $\mathfrak{F}_X$ denotes the $\sigma$-algebra on $X$.
\begin{definition}
The \textit{index of compactness} of an operator $\mathcal{L}$ denoted $\alpha(\mathcal{L})$, is the infimum of those real numbers $t$ such that the image of the unit ball in $X$ under $\mathcal{L}$ may be covered by finitely many balls of radius $t$.
\label{def:IndexOfCompactness}
\end{definition}\noindent
The index of compactness provides a notion of `how far' an operator is from being compact. {This definition was extended by Thieullen to random compositions of operators in \cite{IOC_TP}.}
\begin{definition}
The \textit{asymptotic index of compactness} for the cocycle $(\mathcal{L}_\omega)_{\omega\in\Omega}$ is
$$\kappa(\omega)=\lim_{n\to \infty} \frac{1}{n}\log \alpha(\mathcal{L}_\omega^{(n)}).$$
\label{def:AssymptoticIndexOfCompactness}
\end{definition}\noindent
We call the cocycle \textit{quasi-compact} if $\kappa<\lim_{n\to \infty} \frac{1}{n}\log||\mathcal{L}_\omega^{(n)}||=:\lambda_1(\omega)$, whose limit exists for $\mathbb{P}$-a.e. $\omega\in\Omega$ {and it is independent of $\omega$}, by the Kingman sub-additive ergodic theorem {\cite{Kingman}}, under the assumption that $\int \log||\mathcal{L}_\omega||\, d\mathbb{P}(\omega)<\infty$. The limit $\lambda_1(\omega)$ is referred to as the top Lyapunov exponent of the Perron-Frobenius operator cocycle, and under some assumptions on the random dynamical system, we can obtain a spectrum of these exponents through {multiplicative ergodic theorems}. One example is \textit{Oseledets decomposition} which splits our space of densities into $\omega$ dependent subspaces which decay/expand according to its associated Lyapunov exponent $\lambda_i(\omega)$. These are constant $\mathbb{P}$-a.e. when $\sigma$ is ergodic.
\begin{theorem}[Oseledets decomposition]
Let $\sigma$ be an invertible ergodic measure-preserving transformation of a probability space $(\Omega,\mathfrak{F},\mathbb{P})$ and let $\omega\mapsto\mathcal{L}_\omega$ be a quasi-compact strongly measurable cocycle of operators acting on a Banach space $X$ with a separable dual satisfying $\int \log||\mathcal{L}_\omega||\, d\mathbb{P}(\omega)<\infty$. Then there exist $1\leq \ell \leq \infty$ exponents $\lambda_1\geq \lambda_2\geq \cdots \geq \lambda_\ell \geq \kappa \geq -\infty$, finite multiplicities $m_1,m_2,\dots ,m_\ell$ and subspaces $V_1(\omega),\dots V_\ell(\omega),W(\omega)$ such that
\begin{itemize}
    \item[(a)] $\dim(V_i(\omega))=m_i$;
    \item[(b)] $\mathcal{L}_\omega V_i(\omega)=V_i(\sigma\omega)$ and $\mathcal{L}_\omega W(\omega)\subset W(\sigma\omega)$;
    \item[(c)] $V_1(\omega)\oplus\cdots \oplus V_\ell(\omega)\oplus W(\omega)=X$;
    \item[(d)] for $f\in V_i(\omega)\setminus\{0\}$, $\lim_{n\to\infty}\frac{1}{n}\log||\mathcal{L}_\omega^{(n)}f||\to \lambda_i$;
    \item[(e)] for $f\in W(\omega)\setminus\{0\}$, $\limsup_{n\to\infty}\frac{1}{n}\log||\mathcal{L}_\omega^{(n)}f||\leq \kappa$.
\end{itemize}
\label{the:OsceledetsDecomp}
\end{theorem}\noindent
We refer to the set of all $\lambda_i$ as the Lyapunov spectrum of the Perron-Frobenius operator cocycle $(\mathcal{L}_\omega)_{\omega\in\Omega}$.

\section{Finite Blaschke products}
\label{sec:fbc}
In this paper, we study the Lyapunov spectrum of Perron-Frobenius operator cocycles for \textit{finite Blaschke products} fixing the origin. Applications of this family of functions are discussed in \cite{PropertiesaceFBP}. We consider first the autonomous case (where the dynamics are independent of $\omega$).\\ \\ \noindent In what follows, we denote the extended complex plane by $\hat{\mathbb{C}}=\mathbb{C}\cup\{\infty\}$, the unit circle by $S^1$, and the open unit disc by $D_1$.
\begin{definition} A \textit{finite Blaschke product} is a map $T:\hat{\mathbb{C}}\to \hat{\mathbb{C}}$ of the form
\begin{equation}
    T(z) = \rho \prod_{i=1}^n \frac{z-\zeta_i}{1-\bar{\zeta}_iz},
    \label{eqn:nonrandBP}
\end{equation}
where $\rho\in S^1$ and for each $1\leq i\leq n$, $\zeta_i\in D_1$.
The number $n\in \mathbb N$ is called the \textit{degree} of $T$.
The set of finite Blaschke products will be denoted by $\mathfrak B$, and the set of finite Blaschke products of degree $n$ will be denoted by $\mathfrak B_n$.
\label{def:Blaschke product}
\end{definition}

\begin{remark}
When $\rho=1$, \eqref{eqn:nonrandBP} is called a \textit{monic} finite Blaschke product.
\end{remark}
\noindent Finite Blaschke products enjoy many analytical properties described in \cite{PropertiesaceFBP,PropertiesdFBP}, including the following.
\begin{lemma}[Properties of finite Blaschke products]
Let $T$ be a finite Blaschke product. Then
\begin{itemize}
    \item[(a)] $T$ maps the unit circle to itself $(T(S^1)=S^1)$;
    \item[(b)] $T\circ I = I \circ T$, where $I$ is the inversion map $I(z)=1/\bar{z}$;
    \item[(c)] $T$ maps the open unit disk $D_1$ to itself $($and hence maps $\hat{\mathbb{C}}\setminus \bar{D}_1$ to itself$)$;
    \item[(d)] If $T$ is a non-constant map from the closed unit disk to itself that is analytic in the interior and maps the boundary to itself, then $T$ is a finite Blaschke product;
    \item[(e)] The composition of two finite Blaschke products is again a finite Blaschke product.
\end{itemize}
\label{lem:Properties of finite Blaschke products}
\end{lemma}
\noindent In light of the work completed by \cite{gonzaleztokman2018stability}, we restrict ourselves to {expanding} finite Blaschke products when constrained to $S^1$ ($|T^\prime(z)|>1$ for all $z\in S^1$). {Such systems possess many interesting properties, especially when we consider the random setting.}
Taking $r_T(R)=\max_{|z|=R}|T(z)|$, \cite{Tischlerexpand} illustrates that a necessary condition for $T$ to be expanding is that $r_T(R)<R$ for some $0<R<1$. Further, \cite{martin1983expand} provides a sufficient condition for the expansion of a Blaschke product on $S^1$ dependant only on the zeros of the map. In particular, if $\sum_{i=1}^n\frac{1-|\zeta_i|}{1+|\zeta_i|}>1$, then \cite{martin1983expand} illustrates that $T$ is expanding on $S^1$.
\\ \\ \noindent As opposed to studying a single map $T$, we now consider the {random} (non-autonomous) case and let $T_\omega$ be selected by a driving system $\sigma:\Omega\to\Omega$. {In this case, after $n$ iterates, a  point $z\in \hat{\mathbb{C}}$ initialised at the \textit{environment} $\omega\in \Omega$ is mapped to}
$$T_\omega^{(n)}(z):= (T_{\sigma^{n-1}\omega}\circ \cdots \circ T_\omega)(z).$$
\begin{definition}
A \textit{random Blaschke product} or \textit{Blaschke product cocycle} consists of an ergodic invertible measure-preserving transformation $\sigma$ of $(\Omega,\mathfrak{F},\mathbb{P})$ and
 a map $\mathcal{T}_{(n,\rho,\zeta)}:\Omega\times \hat{\mathbb{C}}\to \hat{\mathbb{C}}$, parametrised by $n, \ \rho$ and $\zeta=(\zeta_{1},\cdots,\zeta_{n})$, of the form
\begin{equation} \mathcal{T}_{(n,\rho,\zeta)}(\omega, z)=:T_\omega(z)=\rho_\omega \prod_{i=1}^{n_\omega} \frac{z-\zeta_{i,\omega}}{1-\bar{\zeta}_{i,\omega}z},
    \label{eqn:RFBP}
\end{equation}
where $n:\Omega \to \mathbb{N}\setminus \{0,1\}$ and
$\rho:\Omega \to S^1$ are measurable, and  for each $m\geq 2$ such that $\Omega_m:=\{\omega \in \Omega \ | \ n_\omega = m\}$ is non-empty, the function $\zeta:\Omega_m\to D_1^m$ is measurable. When $n$ is a constant function, we say the cocycle is a \textit{Blaschke product cocycle of degree $n$}.
\end{definition}
\begin{remark}
    {We exclude $n_\omega = 1$, since in that case, the cocycle is not uniformly expanding on $S^1$.}
\end{remark}
\begin{remark}
     We often refer to  random Blaschke products as in \eqref{eqn:RFBP} as $\mathcal{T}_{(n,\rho,\zeta)}$ or $(T_\omega)_{\omega\in\Omega}$, disregarding the notation  of the explicit dependence on $\sigma$ (or even $(n,\rho,\zeta)$) when the context is clear.
\end{remark}

\begin{definition}
    The \textit{space of random Blaschke products}, denoted $\mathfrak{B}(\Omega)$, consists of all random Blaschke products $\mathcal{T}_{(n,\rho,\zeta)}$ as in \eqref{eqn:RFBP}.
\end{definition}

 \begin{definition}
  The \textit{space of Blaschke product cocycles of degree $n$}, denoted $\mathfrak{B}_{n}(\Omega)$, consists of random Blaschke products ${\mathcal{T}_{(n,\rho,\zeta)}}=(T_\omega)_{\omega\in\Omega}$ with $n:\Omega\to \mathbb{N}\setminus \{0,1\}$ satisfying $n_\omega = n$ for every $\omega\in\Omega$. When $n$ is understood, cocycles belonging to $\mathfrak{B}_{n}(\Omega)$ are denoted $\mathcal{T}_{(\rho,\zeta)}$.
 \end{definition}\noindent
 Of particular interest to us are those Blaschke product cocycles fixing the origin.
  \begin{definition}
  \label{def:fbp_fixorig}
 The \textit{space of Blaschke product cocycles fixing the origin}, denoted $\accentset{\circ}{\mathfrak{B}}(\Omega)$, consists of those $\mathcal{T}_{(n,\rho,\zeta)} \in {\mathfrak{B}}(\Omega)$ such that for each $\omega\in\Omega$, there exists $i=1,\dots, n$ such that $\zeta_{i,\omega}=0$.
 Similarly, the \textit{space of Blaschke product cocycles of degree $n$ fixing the origin}, denoted $\accentset{\circ}{\mathfrak{B}}_{n}(\Omega)$, consists of  all $\mathcal{T}_{(\rho,\zeta)} \in {\mathfrak{B}}_{n}(\Omega) \cap \accentset{\circ}{\mathfrak{B}}(\Omega)$.
 \end{definition}

 \begin{remark}
     If $\mathcal{T}_{(\rho,\zeta)}\in \accentset{\circ}{\mathfrak{B}}_{n}(\Omega)$, without loss of generality we assume that $\zeta=(0,\zeta_{2},\dots,\zeta_{n})$. In particular, if $\mathcal{T}_{(\rho,\zeta)}=(T_\omega)_{\omega\in\Omega}\in \accentset{\circ}{\mathfrak{B}}_{n}(\Omega)$ then we take
$${T}_\omega(z) = \rho_\omega z\prod_{i=2}^n\frac{z-\zeta_{i,\omega}}{1-\bar{\zeta}_{i,\omega}z}.$$
This assumption is valid as the cocycle $\mathcal{T}_{(\rho,\zeta)}$ is invariant under permutations of the coordinates of $\zeta$.
\label{rmk:fbp_fix0}
 \end{remark}\noindent
 As discussed in \cite{gonzaleztokman2018stability}, {which generalises the necessary condition of expansion by \cite{Tischlerexpand} to the random setting, if $r_{\mathcal{T}_{(n,\rho,\zeta)}}(R):=\esssup_{\omega \in \Omega}r_{{T}_\omega}(R)<R$ for some $0<R<1$ {(where $r_{T_\omega}(R)=\max_{|z|=R}|T_\omega(z)|$)} then} ${\mathcal{T}_{(n,\rho,\zeta)}}$ is expanding when restricted to $S^1$ for each $\omega\in\Omega$. If $\mathcal{T}_{(n,\rho,\zeta)}$ satisfies $r_{\mathcal{T}_{(n,\rho,\zeta)}}(R)<R$ for some $0<R<1$, we refer to $\mathcal{T}_{(n,\rho,\zeta)}$ as an \textit{admissible Blaschke product cocycle}.

 \begin{definition}
{Let $0<R<1$. The space of all cocycles ${\mathcal{T}_{(n,\rho,\zeta)}}=(T_\omega)_{\omega\in\Omega}\in\mathfrak{B}_{}(\Omega)$ satisfying $r:=r_{\mathcal{T}_{(n,\rho,\zeta)}}(R)<R$ is called the \textit{space of admissible Blaschke product cocycles}, denoted $\mathfrak{B}_{R}^a(\Omega)$. Similarly, the space of all cocycles ${\mathcal{T}_{(n,\rho,\zeta)}}=(T_\omega)_{\omega\in\Omega}\in\mathfrak{B}_{n}(\Omega)$ satisfying $r:=r_{\mathcal{T}_{(n,\rho,\zeta)}}(R)<R$ is called the \textit{space of admissible Blaschke product cocycles of degree $n$}, denoted $\mathfrak{B}_{R,n}^a(\Omega)$. }
\label{def:space_of_afbc}
\end{definition}\noindent
If a Blaschke product cocycle is admissible, we have the following results.
\begin{theorem}[Lyapunov spectrum of a Blaschke product cocycle \cite{gonzaleztokman2018stability}]
Let $\sigma$ be an invertible ergodic measure-preserving transformation of a probability space $(\Omega,\mathfrak{F},\mathbb{P})$. Let  ${\mathcal{T}_{(n,\rho,\zeta)}}=(T_\omega)_{\omega \in \Omega}$ be an admissible Blaschke product cocycle. Let $\mathcal{L}_\omega$ denote the Perron-Frobenius operator of $T_\omega$, acting on the Hilbert space $H^2(A_R)$ on the annulus $A_R:=\{z \ | \ R<|z|<1/R\}$. Then the cocycle is compact and the following hold:
\begin{itemize}
    \item[(a)](Random Fixed Point). There exists a measurable map $x:\Omega \to \bar{D}_r$ (with $x(\omega)=x_\omega$), such that $T_\omega(x_\omega)=x_{\sigma \omega}$. For all $z\in D_R$, $T^{(n)}_{\sigma^{-n}\omega}(z)=T_{\sigma^{-1}\omega}\circ \cdots \circ T_{\sigma^{-n}\omega}(z)\to x_\omega$;
    \item[(b)](Critical Random Fixed Point). If $\mathbb{P}(\{\omega \ | \ T^\prime_\omega (x_\omega)=0\})>0$, then the Lyapunov spectrum of the cocycle is $0$ with multiplicity $1$; and $-\infty$ with infinite multiplicity.
    \item[(c)](Generic case). If $\mathbb{P}(\{\omega \ | \ T^\prime_\omega (x_\omega)=0\})=0$, then define
    $$\Lambda = \int \log|T^\prime_\omega(x_\omega)|\,d\mathbb{P}(\omega).$$
    This satisfies $\Lambda \leq \log(r/R)<0$. If $\Lambda = -\infty$, then the Lyapunov spectrum of the cocycle is $0$ with multiplicity $1$; and $-\infty$ with infinite multiplicity. If $\Lambda>-\infty$, then the Lyapunov spectrum of the cocycle is $0$ with multiplicity $1$; and $n\Lambda$ with multiplicity $2$ for each $n\in \mathbb{N}$. The Oseledets space with exponent $0$ is spanned by ${1}/{(z-x_\omega)}-{1}/{(z-{1}/{\bar{x}_\omega})}$. The Oseledets space with exponent $j\Lambda$ is spanned by two functions, one a linear combination of ${1}/{(z-x_\omega)^2},\cdots,{1}/{(z-x_\omega)^{j+1}}$ with a pole of order $j+1$ at $x_\omega$; the other a linear combination of ${z^{k-1}}/{(1-\bar{x}_\omega z)^{k+1}}$ for $k=1,\dots,j$, with a pole of order $j+1$ at ${1}/{\bar{x}_\omega}$.
\end{itemize}
\label{the:Lyapunov Spectrum of a Blaschke product cocycle}
\end{theorem}
\begin{theorem}[Stability of Lyapunov Spectrum \cite{gonzaleztokman2018stability}]
Let $\sigma$ be an ergodic invertible measure-preserving transformation of $(\Omega,\mathfrak{F},\mathbb{P})$. Let  $(T_\omega)_{\omega \in \Omega}$ be an admissible Blaschke product cocycle.
\begin{enumerate}
    \item \label{thm:stable}
    Suppose $\essinf_{\omega\in \Omega}|T^\prime_\omega(x_\omega)|>0$. Then if $(\mathcal{L}_\omega)$ is the Perron-Frobenius cocycle of $(T_\omega)_{\omega\in\Omega}$ and $(\mathcal{L}_\omega^\varepsilon)_{\omega\in\Omega}$ is a family of Perron-Frobenius cocycles such that $\esssup_{\omega \in \Omega}\|\mathcal{L}_{\omega}^\varepsilon - \mathcal{L}_\omega \| \to 0$ as $\varepsilon \to 0$, then $\mu_k^\varepsilon \to \mu_k$ as $\varepsilon \to 0$, where $(\mu_k)$ is the sequence of Lyapunov exponents of $(\mathcal{L}_\omega)$, listed with multiplicity and $(\mu_k^\varepsilon)$ is the sequence of Lyapunov exponents of $(\mathcal{L}_\omega^\varepsilon)_{\omega\in\Omega}$.
    \item \label{thm:unstable}
    Suppose $\essinf_{\omega\in \Omega}|T^\prime_\omega(x_\omega)|=0$. Then there exists a family of Blaschke product cocycles $(T^\varepsilon_\omega)_{\omega \in \Omega}$ such that $\esssup_{\omega \in \Omega}\|\mathcal{L}_{T_\omega^\varepsilon} - \mathcal{L}_{T_\omega} \| \to 0$ as $\varepsilon \to 0$ with the property that the Lyapunov exponents of $(\mathcal{L}_{T_\omega^\varepsilon})_{\omega\in\Omega}$ are $0$ with multiplicity $1$, and $-\infty$ with infinite multiplicity for all $\varepsilon>0$.
\end{enumerate}
\label{the:Stability of Lyapunov Spectrum}
\end{theorem}
\begin{remark}
\sloppy   { In relation to the statement of \thrm{the:Stability of Lyapunov Spectrum}, in the case that $\essinf_{\omega\in \Omega}|T^\prime_\omega(x_\omega)|=0$ we say that the cocycle is unstable. In this case, the unperturbed system may have a complete Lyapunov spectrum, however, arbitrarily small perturbations can lead to a collapse. In the other cases, we call the cocycle stable. }
\end{remark}

\begin{definition}
 Denote by ${\mathfrak{B}}_{R}^{a,S}(\Omega)$, respectively ${\mathfrak{B}}_{R}^{a,U}(\Omega)$, the \textit{space of stable and unstable admissible Blaschke product cocycles} satisfying, respectively, the stability or instability
condition of Theorem~\ref{the:Stability of Lyapunov Spectrum}.

\end{definition}\noindent
Then, we have the following.
\begin{corollary}
\begin{align*}&{\mathfrak{B}}_{R}^{a,S}(\Omega)=\left\{ {\mathcal{T}_{(n,\rho,\zeta)}}\in{\mathfrak{B}}_{R}^a(\Omega) \ \big| \  \essinf_{\omega\in \Omega}|{T}^\prime_\omega(x_\omega)|>0\right\},\\&{\mathfrak{B}}_{R}^{a,U}(\Omega)=\left\{ {\mathcal{T}_{(n,\rho,\zeta)}}\in{\mathfrak{B}}_{R}^a(\Omega) \ \big| \  \essinf_{\omega\in \Omega}|{T}^\prime_\omega(x_\omega)|=0\right\}.\end{align*}
\end{corollary}\noindent
We restrict our analysis to those cocycles for which \thrm{the:Lyapunov Spectrum of a Blaschke product cocycle} and \thrm{the:Stability of Lyapunov Spectrum} apply and hence only consider admissible cocycles.
\begin{remark}
In what follows, we will use subscripts and superscripts to denote various subsets of the spaces $\accentset{\circ}{\mathfrak{B}}(\Omega)$ and ${\mathfrak{B}}(\Omega)$.
In particular, the superscript $a$ refers to admissible cocycles, $m$ refers to monic cocycles (studied in \Sec{sec:monicquad}), and $S,U$ is used to specify spaces of stable and unstable cocycles, respectively. The subscript $R$ is used in the case of admissible cocycles when elements must satisfy the condition $r<R$ (we will find in \cor{cor:BaeqB} that this notation is not essential when studying $\accentset{\circ}{\mathfrak{B}}(\Omega)$). Finally, unless specified by the subscript $n$, we consider cocycles of varying degree. In the case that the subscript $n$ is used, we consider cocycles of fixed degree $n$.
\end{remark}

\section{Blaschke product cocycles fixing the origin}
\label{sec:fbc0}
 Our main results concern admissible Blaschke product cocycles of degree $n$ fixing the origin. In this section, we show that all Blaschke product cocycles of degree $n$ fixing the origin are admissible. Then, through \thrm{the:Stability of Lyapunov Spectrum}, we establish stability conditions for such cocycles.\\ \\ \noindent Studying this family of Blaschke product cocycles is of interest, as every $\mathcal{T}_{(\rho,\zeta)} \in \mathfrak{B}_{R,n}^a(\Omega)$ may be conjugated to a cocycle $\hat{\mathcal{T}}_{(\hat{\rho},\hat{\zeta})}\in \accentset{\circ}{\mathfrak{B}}_{n}(\Omega)$ fixing the origin, with stability (in the sense of \thrm{the:Stability of Lyapunov Spectrum}) being preserved. This conjugation procedure is described in \cite{gonzaleztokman2018stability}. {
The difficulty in generalising the results of this work to cocycles that do not necessarily fix the origin lies in establishing smoothness of the coefficient functions parameterising the conjugated Blaschke product cocycle.

\begin{definition}
The \textit{space of admissible Blaschke product cocycles of degree $n$ fixing the origin} denoted $\accentset{\circ}{\mathfrak{B}}_{R,n}^a(\Omega)$ consists of $\mathcal{T}_{(\rho,\zeta)} \in {\mathfrak{B}}_{R,n}^a(\Omega)$ such that for {$\mathbb P$-a.e.} $\omega\in\Omega$, there exists a $i\in\{1,\dots, n\}$ such that $\zeta_{i,\omega}=0$ .
\label{dfnn:AFPBC}
\end{definition}

\begin{remark}
As we will find in \cor{cor:BaeqB}, the $R$ dependence in $\accentset{\circ}{\mathfrak{B}}_{R,n}^a(\Omega)$ can be disregarded since for all $\mathcal{T}_{(\rho , \zeta)}\in \accentset{\circ}{\mathfrak{B}}_n(\Omega)$, $\mathcal{T}_{(\rho , \zeta)}$ is admissible.
\end{remark}
\begin{remark}
 If $\mathcal{T}_{(\rho,\zeta)}\in\accentset{\circ}{\mathfrak{B}}_{R,n}^a(\Omega)$ we follow the same convention as described in \rem{rmk:fbp_fix0}, and assume $\zeta_{1,\omega}=0$ for {$\mathbb P$-a.e.} $\omega\in\Omega$.
\end{remark}
\begin{proposition}[Estimate on $r_{{T}_\omega}(R)$]\label{prop:est rT}
Suppose that $\mathcal{T}_{(\rho,\zeta)}\in\accentset{\circ}{\mathfrak{B}}_{R,n}(\Omega)$ and fix $R<1$. Define $r_{{T}_\omega}(R)=\max_{|z|=R}|{T}_\omega(z)|$. Then
\begin{equation}
    r_{{T}_\omega}(R)\leq R\prod_{i=2}^n \frac{R+|\zeta_{i,\omega}|}{R|\zeta_{i,\omega}|+1}=:M(|\zeta_{2,\omega}|,\cdots,|\zeta_{n,\omega}|).
    \label{eqn:est rT}
\end{equation}
\begin{proof}
Let $z=x+iy$, and $\zeta_{i,\omega} = u_{i,\omega}+i v_{i,\omega}$ where $u_{i},v_{i}:\Omega \to (-1,1)$. Then for $|z|=\sqrt{x^2+y^2}=R$,
\begin{align}
 r_{{T}_\omega}(R)&= \max_{|z|=R}\left|\rho_{\omega}z\prod_{i=2}^n\frac{z-\zeta_{i,\omega}}{1-\bar{\zeta}_{i,\omega}z}\right| \nonumber\\
 &\leq R \max_{x\in [-R,R]} \prod_{i=2}^n\sqrt{\frac{{R^2+|\zeta_{i,\omega}|^2-2u_{i,\omega} x\pm 2v_{i,\omega}\sqrt{R^2-x^2}}}{{R^2|\zeta_{i,\omega}|^2+1-2u_{i,\omega} x\pm 2v_{i,\omega}\sqrt{R^2-x^2}}}} \nonumber \\
 &=:R\max_{x\in [-R,R]}\prod_{i=2}^nf_{i,\omega}^{\pm}(x).
 \label{eqn:fpm def}
\end{align}
It suffices to show that for each $i=2,\dots,n$ that
$$f_{i,\omega}^\pm(x)\leq \frac{R+|\zeta_{i,\omega}|}{R|\zeta_{i,\omega}|+1}$$
for every $\omega \in \Omega$. Let $f_{i,\omega}^\pm (x)=\sqrt{g_{i,\omega}^\pm (x)}$. Elementary calculation shows that $\partial_xg_{i,\omega}^\pm(x)$ does not exist for $x=\pm R$ and thus $\partial_x f_{i,\omega}^\pm$ exists for $x\in (-R,R)$. When $\zeta_{i,\omega} \in D_1$ and $R<1$, $\partial_x f_{i,\omega}^\pm(x) = 0$ if and only if \begin{equation}
    \left(R^2-1\right) \left(|\zeta_{i,\omega}|^2-1\right) \left(u_{i,\omega} \sqrt{R^2-x^2}\pm v_{i,\omega} x\right) = 0
    \label{eqn:fprime}
\end{equation} First suppose that $|\zeta_{i,\omega}|=0$ for every $i=2,\dots,n$, \eqref{eqn:est rT} holds when $r_{T_\omega}(R)\leq R^n$. Indeed, when $|\zeta_{i,\omega}|=0$, $f_{i,\omega}^\pm(x)=R$ for each $i=2,\dots,n$ and every $\omega\in\Omega$. Thus by \eqref{eqn:fpm def}, $r_{T_\omega}(R)\leq R^n$ and \eqref{eqn:est rT} is satisfied. Now suppose that $|\zeta_{i,\omega}|\neq0$. Since $R,|\zeta_{i,\omega}|<1$, \eqref{eqn:fprime} holds if and only if
$$x^*_\pm = \pm \frac{R u_{i,\omega} }{|\zeta_{i,\omega}|}.$$
The function $f_{i,\omega}^\pm(x)$ is continuous on $[-R,R]$ and differentiable on $(-R,R)$, so its maximum/minimum value is attained at $x^*_\pm$, or on the boundary of the domain where $x=\pm R$. Critical points of $f_{i,\omega}^\pm$ are also critical points of $(f_{i,\omega}^\pm(x))^2$ and their nature is preserved since $f_{i,\omega}^\pm(x)\geq 0$. It therefore suffices to determine the maximum value of $(f_{i,\omega}^\pm(x))^2$ determined by 8 candidate values
\begin{align*}
    (f_{i,\omega}^\pm\left(R\right))^2&=\frac{R^2+|\zeta_{i,\omega}|^2-2Ru_{i,\omega}}{R^2|\zeta_{i,\omega}|^2+1-2Ru_{i,\omega}}\\[2mm]
    (f_{i,\omega}^\pm\left(-R\right))^2&=\frac{R^2+|\zeta_{i,\omega}|^2+2Ru_{i,\omega}}{R^2|\zeta_{i,\omega}|^2+1+2Ru_{i,\omega}}\\[2mm]
    \left(f_{i,\omega}^\pm\left(\frac{Ru_{i,\omega}}{|\zeta_{i,\omega}|}\right)\right)^2&=\frac{{R^2+|\zeta_{i,\omega}|^2+\frac{2R}{|\zeta_{i,\omega}|}(-u_{i,\omega}^2\pm v_{i,\omega}|v_{i,\omega}|)}}{{R^2|\zeta_{i,\omega}|^2+1+\frac{2R}{|\zeta_{i,\omega}|}(-u_{i,\omega}^2\pm v_{i,\omega}|v_{i,\omega}|)}}\\[2mm]
    \left(f_{i,\omega}^\pm\left(-\frac{Ru_{i,\omega}}{|\zeta_{i,\omega}|}\right)\right)^2&=\frac{{R^2+|\zeta_{i,\omega}|^2+\frac{2R}{|\zeta_{i,\omega}|}(u_{i,\omega}^2\pm v_{i,\omega}|v_{i,\omega}|)}}{{R^2|\zeta_{i,\omega}|^2+1+\frac{2R}{|\zeta_{i,\omega}|}(u_{i,\omega}^2\pm v_{i,\omega}|v_{i,\omega}|)}}.
\end{align*}
The above expressions are of the form
$$h_\pm(w(u_{i,\omega},v_{i,\omega}))=\frac{{R^2+|\zeta_{i,\omega}|^2+w(u_{i,\omega},v_{i,\omega})}}{{R^2|\zeta_{i,\omega}|^2+1+w(u_{i,\omega},v_{i,\omega})}}$$
for some $w(u_{i,\omega},v_{i,\omega}):D_1\times D_1 \to \mathbb{R}$. Elementary computation shows that $h_\pm(w(u_{i,\omega},v_{i,\omega}))$ is a strictly increasing function in $w(u_{i,\omega},v_{i,\omega})$. The maximum value of $(f_{i,\omega}^\pm(x))^2$ corresponds to the candidate for which $w(u_{i,\omega},v_{i,\omega})$ is the largest.\\ \\ \noindent Fix $v_{i,\omega}\in(-1,0]$ and $u_{i,\omega}\in(-1,1)$, then $w(u_{i,\omega},v_{i,\omega})=\frac{2R}{|\zeta_{i,\omega}|}(u_{i,\omega}^2-v_{i,\omega}|v_{i,\omega}|)$ is largest or equal to other candidate extrema. For these parameter values, $f_{i,\omega}^-(x_-^*)$ is the global maximum of $f_{i,\omega}^\pm(x)$. When $v_{i,\omega}\in[0,1)$ and $u_{i,\omega}\in(-1,1)$, $f_{i,\omega}^+(x_-^*)$ corresponds to the global maximum of $f_{i,\omega}^\pm(x)$. Computing the maximum values attained,
\begin{align*}
    f_{i,\omega}^\pm(x^*_-)&=\sqrt{\frac{{R^2+|\zeta_{i,\omega}|^2+\frac{2R}{|\zeta_{i,\omega}|}(u_{i,\omega}^2\pm v_{i,\omega}|v_{i,\omega}|)}}{{R^2|\zeta_{i,\omega}|^2+1+\frac{2R}{|\zeta_{i,\omega}|}(u_{i,\omega}^2\pm v_{i,\omega}|v_{i,\omega}|)}}}.
\end{align*}
For $u_{i,\omega} \in (-1,1)$, taking $f_{i,\omega}^-$ and $f_{i,\omega}^+$ for $v_{i,\omega}\in(-1,0]$ and $v_{i,\omega}\in[0,1)$ respectively, we can deduce that $\max_{x\in(-1,1)}f_{i,\omega}^\pm(x) = (R+|\zeta_{i,\omega}|)/(R|\zeta_{i,\omega}|+1)$.
\end{proof}
\end{proposition}\noindent
With this, one may establish an estimate on $r_{{\mathcal{T}}_{(\rho,\zeta)}}(R)$.
\begin{proposition}[Estimate on $r_{{\mathcal{T}}_{(\rho,\zeta)}}(R)$]
   Fix $R<1$ and suppose that $\mathcal{T}_{(\rho,\zeta)}\in\accentset{\circ}{\mathfrak{B}}_{R,n}(\Omega)$. Define $\zeta_{i}^*=\esssup_{\omega\in\Omega}|\zeta_{i,\omega}|$, then
$$r_{\mathcal{T}_{(\rho,\zeta)}}(R)\leq M(|\zeta_{2}^*|,\cdots,|\zeta_{n}^*|),$$
where $M$ is as in \eqref{eqn:est rT}. Further, if
\begin{equation}
    \prod_{i=2}^n\frac{R+|\zeta_{i}^*|}{R|\zeta_{i}^*|+1}<1
    \label{eqn:estimate on rmathcalTw for general}
\end{equation}
then $r_{{\mathcal{T}}_{(\rho,\zeta)}}(R)<R$.
\begin{proof}
    To show that $r_{\mathcal{T}_{(\rho,\zeta)}}(R)\leq M(|\zeta_{2}^*|,\cdots,|\zeta_{n}^*|)$, it suffices to show that $M(|\zeta_{2,\omega}|,\cdots, |\zeta_{n,\omega}|)$ is a monotone increasing function in $|\zeta_{2,\omega}|,\dots,|\zeta_{n,\omega}|$. Indeed
    \begin{align*}
        \nabla M(|\zeta_{2,\omega}|,\cdots, |\zeta_{n,\omega}|)&= \Bigg( \frac{1-R^2}{(R|\zeta_{2,\omega}|+1)^2}\prod_{i=3}^n\frac{R+|\zeta_{i,\omega}|}{R|\zeta_{i,\omega}|+1},\\
        &\quad \ \cdots, \frac{1-R^2}{(R|\zeta_{n,\omega}|+1)^2}\prod_{i=2}^{n-1}\frac{R+|\zeta_{i,\omega}|}{R|\zeta_{i,\omega}|+1} \Bigg),
    \end{align*}
    which has positive entries for every $\omega \in \Omega$. Thus, using Proposition~\ref{eqn:est rT}, for $r_{{\mathcal{T}}_{(\rho,\zeta)}}(R)\leq M(|\zeta_{2}^*|,\cdots,|\zeta_{n}^*|)<R$ it is sufficient to have
    $$\prod_{i=2}^n \frac{R+|\zeta_{i}^*|}{R|\zeta_{i}^*|+1}<1,$$
        as claimed.
\end{proof}
\label{prop:admis_cond}
\end{proposition}
\begin{corollary}
For every $R<1$, $\accentset{\circ}{\mathfrak{B}}_{R,n}^a(\Omega)=\accentset{\circ}{\mathfrak{B}}_{n}(\Omega)$.
\begin{proof}
    Take $\mathcal{T}_{(\rho,\zeta)}\in\accentset{\circ}{\mathfrak{B}}_{n}(\Omega)$. It suffices to show that $$\frac{R+|\zeta_{i}^*|}{R|\zeta_{i}^*|+1}<1$$
since \prop{prop:admis_cond} will imply that $r_{{\mathcal{T}}_{(\rho,\zeta)}}(R)<R$.
Indeed, since $|\zeta_{i}^*|<1$ we get
\begin{alignat*}{3}
&\phantom{\implies} & \frac{R+|\zeta_{i}^*|}{R|\zeta_{i}^*|+1}&<1\\
&\iff & R+|\zeta_{i}^*|&<R|\zeta_{i}^*|+1\\
&\iff & R(1-|\zeta_{i}^*|)&<1-|\zeta_{i}^*|\\
&\iff & R&<1.
\end{alignat*}
   This inequality is always satisfied since $R<1$. In turn, all Blaschke product cocycles of degree $n$ fixing the origin are admissible {for every $R<1$}.
\end{proof}
\label{cor:BaeqB}
\end{corollary}\noindent
As a result of \cor{cor:BaeqB} and for the sake of notation, we denote from here onwards the space of admissible Blaschke product cocycles of degree $n$ fixing the origin as $\accentset{\circ}{\mathfrak{B}}_{n}(\Omega)$.
\\ \\ \noindent
Take ${\mathcal{T}}_{(\rho,\zeta)}\in \accentset{\circ}{\mathfrak{B}}_{n}(\Omega)$, then through \thrm{the:Stability of Lyapunov Spectrum} we can derive stability conditions for ${\mathcal{T}}_{(\rho,\zeta)}$.
\begin{proposition}[Stability conditions for ${\mathcal{T}}_{(\rho,\zeta)}$]
Suppose that ${\mathcal{T}}_{(\rho,\zeta)} \in \accentset{\circ}{\mathfrak{B}}_{n}(\Omega)$. Then $\mathcal{T}_{(\rho,\zeta)}$ is:
\begin{enumerate}
    \item Stable if
    $$\essinf_{\omega\in\Omega}\left|\prod_{i= 2}^n\zeta_{i,
    \omega}\right|>0.$$
    \item Unstable if
    $$\essinf_{\omega\in\Omega}\left|\prod_{i= 2}^n\zeta_{i,
    \omega}\right|=0.$$
\end{enumerate}
\begin{proof}
    We compute ${T}^\prime_\omega(z)$. Here,
    \begin{align*}
        \rho_\omega^{-1}{T}^\prime_\omega(z)&=\prod_{i=2}^n\frac{z-\zeta_{i,\omega}}{1-\bar{\zeta}_{i,\omega}z} + z \frac{d}{dz}\prod_{i=2}^n\frac{z-\zeta_{i,\omega}}{1-\bar{\zeta}_{i,\omega}z}\\
        &=\left(\prod_{i=2}^n\frac{z-\zeta_{i,\omega}}{1-\bar{\zeta}_{i,\omega}z}\right)\left(1 + z\sum_{i=2}^n \frac{1-|\zeta_{i,\omega}|^2}{(z-\zeta_{i,\omega})(1-\bar{\zeta}_{i,\omega}z)}\right).
    \end{align*}
    The stability conditions from \thrm{the:Stability of Lyapunov Spectrum} require us to evaluate ${T}_\omega^\prime(z)$ at the random fixed point of ${\mathcal{T}}_{(\rho,\zeta)}$, which is given by $x_\omega =0$, since ${\mathcal{T}}_{(\rho,\zeta)}\in \accentset{\circ}{\mathfrak{B}}_{n}(\Omega)$. Thus, since $|\rho_\omega^{-1}|=1$ for every $\omega\in\Omega$
    \begin{align*}
     |{T}^\prime_\omega(0)|&=  \left|\prod_{i=2}^n\zeta_{i,\omega}\right|.
    \end{align*}
    A direct application of \thrm{the:Stability of Lyapunov Spectrum} gives us our desired result.
\end{proof}
\label{prop:stab_instab}
\end{proposition}
\noindent For technical reasons which will become apparent in \Sec{sec:prev_proof}, we restrict ourselves to cocycles parametrised by continuously differentiable coefficient functions. More precisely, we take $\zeta \in C^1(\Omega,\{0\} \times D_{1}^{n-1})$, {where $\Omega$ is assumed to be a $C^1$ manifold.}
\begin{corollary}
Let ${\mathcal{T}}_{(\rho,\zeta)}\in\accentset{\circ}{\mathfrak{B}}_{n}(\Omega)$ with $\zeta \in C^1(\Omega, \{0\} \times D_{1}^{n-1})$ {and suppose $\mathbb P$ has full support}. Then, ${\mathcal{T}}_{(\rho,\zeta)}$ is unstable if and only if there exists $i=2,\dots,n$ and $\omega\in\Omega$ such that $\zeta_{i,\omega}=0$.
\begin{proof}
By \prop{prop:stab_instab}, ${\mathcal{T}}_{(\rho,\zeta)}$ is unstable if and only if
$$\essinf_{\omega\in\Omega}\left|\prod_{i= 2}^n\zeta_{i,
    \omega}\right|=0.$$
    Since $\zeta_{i}\in C^1(\Omega,D_{1})$ {and $\mathbb P$ has full support}, then ${\mathcal{T}}_{(\rho,\zeta)}$ is unstable if and only if
    $$\min_{\omega\in\Omega}\left|\prod_{i= 2}^n\zeta_{i,
    \omega}\right|=0\iff \zeta_{i,
    \omega}=0$$
    for some $i=2,\dots,n$ and $\omega\in\Omega$.
\end{proof}
\label{cor:instab_cond}
\end{corollary}

\section{Prevalence for monic quadratic Blaschke product cocycles fixing zero}
\label{sec:monicquad}
This section addresses prevalence of stability for monic quadratic Blaschke product cocycles fixing the origin. By defining a linear structure on the space of such cocycles we illustrate that this space, equipped with an appropriate metric, is a completely metrisable topological vector space. This allows us to utilise the prevalence framework with the aim to show that almost every monic quadratic Blaschke product cocycle fixing the origin is stable, in the sense of prevalence.

\begin{definition}
The \textit{space of monic quadratic Blaschke product cocycles fixing the origin}, denoted $\accentset{\circ}{\mathfrak{B}}_{2}^m(\Omega)$ consists of all $\mathcal{T}_{(\rho,\zeta)}\in \accentset{\circ}{\mathfrak{B}}_{2}(\Omega)$ such that for every $\omega \in \Omega$, $\rho(\omega)=1$ and $\zeta_1(\omega)=0$.
\end{definition}
\begin{remark}
\label{rem:notationsimp}
   Elements of $\accentset{\circ}{\mathfrak{B}}_{2}^m(\Omega)$ are parameterised solely by $\zeta_2:\Omega \to D_1$. For this reason, and for notation purposes, we write $\mathcal{T}_{\zeta}$ to mean $\mathcal{T}_{(1,(0,\zeta_2))}\in \accentset{\circ}{\mathfrak{B}}_{2}^m(\Omega)$.
\end{remark}
\begin{remark}
By \cor{cor:BaeqB} if $\mathcal{T}_{\zeta}\in \accentset{\circ}{\mathfrak{B}}_{2}^m(\Omega)$, then $\mathcal{T}_{\zeta}$ is admissible.
\end{remark}\noindent
A result of \thrm{the:Stability of Lyapunov Spectrum} is the following.
\begin{corollary}
$$\accentset{\circ}{\mathfrak{B}}_{2}^{m,S}(\Omega)=\left\{ {\mathcal{T}_{\zeta}}\in\accentset{\circ}{\mathfrak{B}}_{2}^m(\Omega) \ \big| \  \essinf_{\omega\in \Omega}|{T}^\prime_\omega(x_\omega)|>0\right\},$$

$$\accentset{\circ}{\mathfrak{B}}_{2}^{m,U}(\Omega)=\left\{ {\mathcal{T}_{\zeta}}\in\accentset{\circ}{\mathfrak{B}}_{2}^m(\Omega) \ \big| \  \essinf_{\omega\in \Omega}|{T}^\prime_\omega(x_\omega)|=0\right\}.$$
\end{corollary}\noindent
We wish to determine whether $\accentset{\circ}{\mathfrak{B}}_{2}^{m,S}(\Omega)$ is prevalent in $\accentset{\circ}{\mathfrak{B}}_{2}^{m}(\Omega)$. By \dfnn{def:probe}, it suffices to show that $\accentset{\circ}{\mathfrak{B}}_{2}^{m,S}(\Omega)$ contains a probe. In our setting, we want to find a finite dimensional subspace $\mathcal{P}\subset \accentset{\circ}{\mathfrak{B}}_{2}^{m}(\Omega)$ so that for all $\mathcal{T}_{\zeta}\in \accentset{\circ}{\mathfrak{B}}_{2}^{m}(\Omega)$, Lebesgue almost every point in the hyperplane $\mathcal{T}_{\zeta}+\mathcal{P}$ belongs to $\accentset{\circ}{\mathfrak{B}}_{2}^{m,S}(\Omega)$. To make sense of the expression $\mathcal{T}_{\zeta}+\mathcal{P}$, we must define what it means to add elements in $\accentset{\circ}{\mathfrak{B}}_{2}^{m}(\Omega)$. Defining a diffeomorphism $\Phi:D_1\to \mathbb{C}$ allows us to do this precisely.
\begin{proposition}[Properties of $\Phi(z)$]
Define the map
\begin{equation}\label{eq:defphi}
\Phi(z) := \frac{z}{\sqrt{1-|z|^2}}.
\end{equation}
Then $\Phi:D_{1}\to \mathbb{C}$ is a diffeomorphism with inverse
\begin{equation}
    \Phi^{-1}(z) = \frac{z}{\sqrt{1+|z|^2}}.
    \label{eq:invphi}
\end{equation}
\begin{proof}
We first show that $\Phi$ is a bijective mapping from $D_1$ to $\mathbb{C}$ with inverse given by \eqref{eq:invphi}. It suffices to show that $(\Phi\circ \Phi^{-1})(z)=(\Phi^{-1}\circ \Phi)(z)=z$. A direct computation of these compositions shows that $\Phi$ is a bijection onto $\mathbb{C}$ with inverse as stated. Indeed for $z\in D_{1}$, we recover a point $\Phi(z)\in\mathbb{C}$.\\ \\ \noindent To show that $\Phi$ is a diffeomorphism, we consider $\Phi$ mapping to a 2-dimensional real manifold. In this case, it suffices to show that $\Phi$ is differentiable in the real sense. Taking $z=x+iy$, we set
$$\Phi(x,y) = \left( \frac{x}{\sqrt{1-x^2-y^2}},\frac{y}{\sqrt{1-x^2-y^2}} \right).$$
For $|z|<1$, this is well defined and the Jacobian matrix is
\begin{equation}
    D\Phi(x,y) =
 \begin{pmatrix}
 \displaystyle \frac{1-y^2}{\left(1-x^2-y^2\right)^{3/2}} & \displaystyle
 \frac{x y}{\left(1-x^2-y^2\right)^{3/2}} \\
 \displaystyle  \frac{x y}{\left(1-x^2-y^2\right)^{3/2}} & \displaystyle
 \frac{1-x^2}{\left(1-x^2-y^2\right)^{3/2}}
\end{pmatrix}.
\label{eqn:DPhi}
\end{equation}
Each component of $D\Phi(x,y)$ is continuous and differentiable for $|z|<1$. Furthermore,
\begin{align*}
    \det(D\Phi(x,y))
    &=\frac{1}{\left(1-x^2-y^2\right)^2}.
\end{align*}
Since $0\leq|z|<1$, $\det(D\Phi(x,y))\neq 0$ for all $(x,y)\in \mathbb{R}^2$, and $\Phi$ is a diffeomorphism.
\end{proof}
\label{prop:PhiRhat(z)}
\end{proposition} \noindent
When $\Phi$ takes vector valued input, we \textit{broadcast} $\Phi$ across each vector element through $\Phi_n:D_1^n \to \mathbb{C}^n$.
\begin{definition}
 Let $z = (z_1,\cdots,z_n)\in D_1^n$ and $\Phi$ be as in \eqref{eq:defphi}. We define the \textit{broadcasting function}\footnote{{This is simply the component-wise extension of the one-coordinate map to higher dimensions.}}  $\Phi_n:D_1^n \to \mathbb{C}^n$ as
 $$\Phi_n(z)=(\Phi(z_1),\cdots, \Phi(z_n))$$
 with inverse $\Phi_n^{-1}:\mathbb{C}^n \to D_1^n$ given by
 $$\Phi_n^{-1}(z)=(\Phi^{-1}(z_1),\cdots, \Phi^{-1}(z_n)).$$
\end{definition} \noindent
With $\Phi_n(z)$, we define addition and scalar multiplication in $\accentset{\circ}{\mathfrak{B}}_{2}^{m}(\Omega)$.
\begin{definition}
Suppose that $\mathcal{T}_{\xi},\mathcal{T}_{\chi}\in\accentset{\circ}{\mathfrak{B}}_{2}^{m}(\Omega)$. We define
$$\mathcal{T}_{\xi}+\mathcal{T}_{\chi}:={\mathcal{T}}_{\Phi_2^{-1}(\Phi_2(\xi)+\Phi_2(\chi))},$$
and for $\alpha \in\mathbb{C}$
$$\alpha\cdot \mathcal{T}_{\xi}:={\mathcal{T}}_{{\Phi^{-1}_2(\alpha \Phi_2(\xi))}}.$$
\label{def:addition_of_fbp}
\end{definition}
\begin{remark}
    Since $\mathcal{T}_{\xi},\mathcal{T}_{\chi}\in\accentset{\circ}{\mathfrak{B}}_{2}^{m}(\Omega)$, then we emphasise that $\xi = (0,\xi_{2}), \chi = (0,\chi_{2})$ for all $\omega \in \Omega$ in the above definition (see {Remark} \ref{rem:notationsimp}).
\end{remark}\noindent
In \fig{fig:maptoC} we illustrate the operation of addition in $\accentset{\circ}{\mathfrak{B}}_{2}^{m}(\Omega)$.

\begin{figure}[H]
\resizebox{\textwidth}{!}{%
\begin{tikzpicture}[
	dot/.style={circle, fill, inner sep=2.15pt},
	myarrow/.style={line width=0.5pt, -{Latex[length=3mm]}}]

\begin{scope}[shift={(-8,0)}]

	\newcommand\CircleRadius{4.5}
	\newcommand\AxisRatio{1.3}

	\fill[Green!30] (0,0) circle (\CircleRadius);
	\draw[Green!70, line width=1.5pt, dashed] (0,0) circle (\CircleRadius);

	\draw[line width=0.5pt, {Latex[length=3mm, width=3mm]}-{Latex[length=3mm, width=3mm]}] (-\AxisRatio*\CircleRadius,0) -- (\AxisRatio*\CircleRadius,0);
	\draw[line width=0.5pt, {Latex[length=3mm, width=3mm]}-{Latex[length=3mm, width=3mm]}] (0,-\AxisRatio*\CircleRadius) -- (0,\AxisRatio*\CircleRadius);

	\node[dot, red] at (0,0) (ori) {};
	\node[dot, inner sep=1pt] at (0,0) {};

	\node[dot, red] at (0,2) (reddot) {};
	\node[dot] at (-1, 1.3) (circleft) {};
	\node[dot] at (1, 1.3) (circright) {};
	\draw[myarrow] (circleft) -- (reddot);
	\draw[myarrow] (circright) -- (reddot);

	\node[below = 1pt of circleft] {\LARGE $\chi_{2, \omega}$};
	\node[below = 1pt of circright] {\LARGE $\xi_{2, \omega}$};
	\node[above = 1pt of reddot] {\LARGE $\chi_{2, \omega}^{\xi_{2, \omega}} = \xi_{2, \omega}^{\chi{2, \omega}}$};
	\node[below = 1pt of ori, xshift=4pt] {\LARGE $\chi_{1, \omega} = \xi_{1, \omega} = \xi_{1, \omega}^{\chi_{1, \omega}} = \chi_{1, \omega}^{\xi_{1, \omega}}$};

	\node at (0, \AxisRatio*\CircleRadius + 0.4) {\huge $D_1$};
\end{scope}

\begin{scope}[shift={(8,0)}]
	
	\newcommand\SquareHalfLength{4.8}
	\newcommand\AxisRatio{1.3}

	\fill[Green!30] (-\SquareHalfLength,-\SquareHalfLength) -- (-\SquareHalfLength,\SquareHalfLength) -- (\SquareHalfLength,\SquareHalfLength) -- (\SquareHalfLength,-\SquareHalfLength) -- (-\SquareHalfLength, -\SquareHalfLength);
	\draw[Green!70, line width=1.5pt, dashed] (-\SquareHalfLength,-\SquareHalfLength) -- (-\SquareHalfLength,\SquareHalfLength) -- (\SquareHalfLength,\SquareHalfLength) -- (\SquareHalfLength,-\SquareHalfLength) -- (-\SquareHalfLength, -\SquareHalfLength);

	\draw[line width=0.5pt, {Latex[length=3mm, width=3mm]}-{Latex[length=3mm, width=3mm]}] (-\AxisRatio*\SquareHalfLength,0) -- (\AxisRatio*\SquareHalfLength,0);
	\draw[line width=0.5pt, {Latex[length=3mm, width=3mm]}-{Latex[length=3mm, width=3mm]}] (0,-\AxisRatio*\SquareHalfLength) -- (0,\AxisRatio*\SquareHalfLength);

	\node[dot, red] at (0,0) (sqori) {};
	\node[dot, inner sep=1pt] at (0,0) {};

	\node[dot, red] at (0,2) (reddot) {};
	\node[dot] at (-1.5, 2) (sqleft) {};
	\node[dot] at (1.5, 2) (sqright) {};
	\draw[myarrow] (sqleft) -- (reddot);
	\draw[myarrow] (sqright) -- (reddot);

	\node[below = 1pt of sqleft] {\LARGE $\Phi_2(\chi_{2, \omega})$};
	\node[below = 1pt of sqright] {\LARGE $\Phi_2(\xi_{2, \omega})$};
	\node[above = 1pt of reddot] {\LARGE $\Phi(\chi_{2, \omega}^{\xi_{2, \omega}}) +  \Phi(\xi_{2, \omega}^{\chi{2, \omega}})$};
	\node[below = 1pt of sqori, xshift=0pt] {\LARGE $\Phi(\chi_{1, \omega}) = \Phi(\xi_{1, \omega}) = \Phi(\xi_{1, \omega}^{\chi_{1, \omega}}) = \Phi(\chi_{1, \omega}^{\xi_{1, \omega}})$};
	
	\node at (0, \AxisRatio*\SquareHalfLength + 0.4) {\huge $\mathbb{C}$};
\end{scope}

\draw[myarrow] (3,-3) to [bend left] (-3,-3);
\node at (0, 4.5) {\huge $\Phi_2$};
\draw[myarrow] (-3,3) to [bend left] (3,3);
\node at (0, -4.5) {\huge $\Phi_2^{-1}$};

\end{tikzpicture}
}
\caption{Addition of cocycles ${\mathcal{T}}_{\xi}$ and ${\mathcal{T}}_{\chi}$. Here we denote ${\xi}_{n}^{\chi_{n}}=\Phi^{-1}(\Phi(\xi_{n})+\Phi(\chi_{n}))={\chi}_{n}^{\xi_{n}}$. The new cocycle ${\mathcal{T}}_{\xi}+{\mathcal{T}}_{\chi}={\mathcal{T}}_{\left({\xi}_{1}^{\chi_{1}},{\xi}_{2}^{\chi_{2}}\right)}={\mathcal{T}}_{\left({\chi}_{1}^{\xi_{1}},{\chi}_{2}^{\xi_{2}}\right)}$.}
    \label{fig:maptoC}
\end{figure}
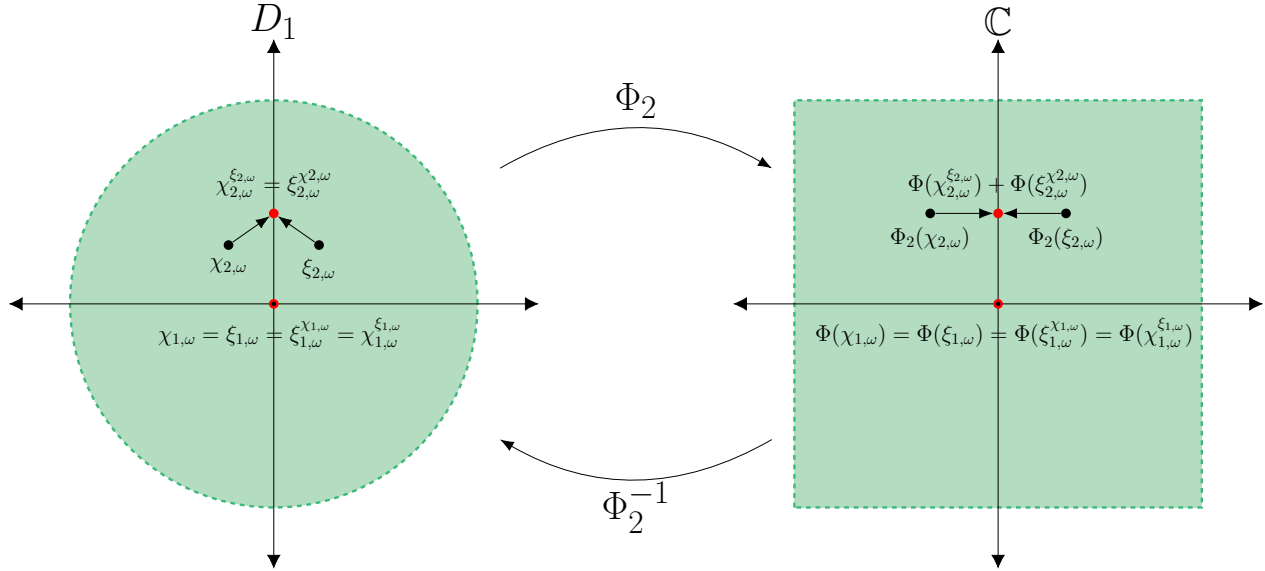
\begin{remark}
    We emphasise that \fig{fig:maptoC} further illustrates that addition in $\accentset{\circ}{\mathfrak{B}}_{2}^{m}(\Omega)$ is commutative in that ${\xi}_{n}^{\chi_{n}}={\chi}_{n}^{\xi_{n}}$.
\end{remark}

\begin{proposition}
{Let $\mathcal{T}_\phi,\mathcal{T}_\chi\in\accentset{\circ}{\mathfrak{B}}_{2}^m(\Omega)$}. The space $(\accentset{\circ}{\mathfrak{B}}_{2}^m(\Omega),+,\cdot)$ equipped with the metric
$$d(\mathcal{T}_{\phi},\mathcal{T}_{\chi})=\esssup_{\omega\in\Omega}|\Phi(\phi_{2,\omega})-\Phi(\chi_{2,\omega})|$$
is a complete topological vector space.
\begin{proof}
An elementary calculation shows that $d$ is indeed a metric on $\accentset{\circ}{\mathfrak{B}}_{2}^m(\Omega)$. We first show that $(\accentset{\circ}{\mathfrak{B}}_{2}^m(\Omega),+,\cdot)$ is complete. Let $(\mathcal{T}_{\phi,k})_{k=1}^\infty\subset \accentset{\circ}{\mathfrak{B}}_{2}^m(\Omega)$ be a Cauchy sequence with respect to $d$, where
\begin{equation}
    \mathcal{T}_{\phi,k}(\omega,z)=:T_{\omega,k}(z) = z \frac{z-\phi_{2,\omega,k}}{1-\bar{\phi}_{2,\omega,k}z}.
    \label{eqn:seq}
\end{equation}
We construct a limiting cocycle $\mathcal{T}_\phi \in \accentset{\circ}{\mathfrak{B}}_{2}^m(\Omega)$ such that $d(\mathcal{T}_{\phi,k},\mathcal{T}_\phi)\to 0$ as $k\to \infty$. If $(\mathcal{T}_{\phi,k})_{k=1}^\infty$ is Cauchy in $d$, then it follows that the sequence $(\Phi(\phi_{2,\omega,k}))_{k=1}^\infty$ is Cauchy in $L^\infty(\mathbb{P})$. Since $L^\infty(\mathbb{P})$ is a Banach space, the sequence $(\Phi(\phi_{2,\omega,k}))_{k=1}^\infty$ must converge to some function $g \in L^\infty(\mathbb{P})$ for $\mathbb{P}$-a.e. $\omega\in\Omega$. Note that $\Phi^{-1}:\mathbb{C}\to D_1$ is a diffeomorphism, thus letting $\Phi^{-1}(g_\omega)=\phi_{2,\omega}\in D_1$, we see that the limiting cocycle of the Cauchy sequence $(\mathcal{T}_{\phi,k})_{k=1}^\infty$ is $\mathcal{T}_{\phi}$, parameterised by $\phi = (0,\phi_{2})$.
Also, clearly $\mathcal{T}_{\phi}$ is a monic quadratic Blaschke product cocycle fixing the origin.
Hence, $\mathcal{T}_{\phi}\in \accentset{\circ}{\mathfrak{B}}_{2}^m(\Omega)$.\\ \\ \noindent We now show that $(\accentset{\circ}{\mathfrak{B}}_{2}^m(\Omega),+,\cdot)$ is a topological vector space. It suffices to prove that the operations $\cdot+\cdot :\accentset{\circ}{\mathfrak{B}}_{2}^m(\Omega)\times \accentset{\circ}{\mathfrak{B}}_{2}^m(\Omega) \to \accentset{\circ}{\mathfrak{B}}_{2}^m(\Omega)$ and $\cdot:\mathbb{C}\times \accentset{\circ}{\mathfrak{B}}_{2}^m(\Omega)\to \accentset{\circ}{\mathfrak{B}}_{2}^m(\Omega)$ are continuous. Suppose that we have sequences $(\mathcal{T}_{\phi,k})_{k=1}^\infty,(\mathcal{T}_{\chi,k})_{k=1}^\infty \subset  \accentset{\circ}{\mathfrak{B}}_{2}^m(\Omega)$ such that $\mathcal{T}_{\phi,k}\to \mathcal{T}_{\phi}$ and  $\mathcal{T}_{\chi,k}\to \mathcal{T}_{\chi}$ in $d$ as $k\to\infty$. Then using \dfnn{def:addition_of_fbp},
\begin{align*} d(\mathcal{T}_{\phi,k}+\mathcal{T}_{\chi,k},\mathcal{T}_{\phi}+\mathcal{T}_{\chi})&= \esssup_{\omega\in\Omega}|\Phi(\phi_{2,\omega,k})+\Phi(\chi_{2,\omega,k})-\Phi(\phi_{2,\omega})-\Phi(\chi_{2,\omega})|\\
&\leq d(\mathcal{T}_{\phi,k},\mathcal{T}_{\phi})+d(\mathcal{T}_{\chi,k},\mathcal{T}_{\chi}).
\end{align*}
This tends to zero in $d$ as $k\to \infty$ and hence addition is continuous. Indeed scalar multiplication is also continuous. Take $\alpha \in\mathbb{C}$. By \dfnn{def:addition_of_fbp}
\begin{align*}
    d(\alpha\mathcal{T}_{\phi,k},\alpha \mathcal{T}_{\phi})&=\esssup_{\omega\in\Omega}|\alpha\Phi(\phi_{2,\omega,k})-\alpha\Phi(\phi_{2,\omega})|\\
    &= |\alpha| d(\mathcal{T}_{\phi,k},\mathcal{T}_{\phi}),
\end{align*}
which tends to zero in $d$ as $k\to \infty$.
\end{proof}
\label{prop:cmtvs}
\end{proposition}

\begin{remark}
 {Let $\mathcal{S}_\chi\in\accentset{\circ}{\mathfrak{B}}_{2}^m(\Omega)$.} In \cite{gonzaleztokman2018stability}, an alternative metric
$$d^\prime(\mathcal{T}_{\phi},\mathcal{S}_{\chi})=\esssup_{\omega\in\Omega}\max_{z\in S^1} |T_\omega(z)-S_\omega(z)|$$
is used. As suggested in \cite{PropertiesaceFBP}, if $\mathcal{T}_\phi \in \accentset{\circ}{\mathfrak{B}}_{2}^m$
\footnote{In our setting, the space $\accentset{\circ}{\mathfrak{B}}_{2}^m$ could be regarded as those $\mathcal{T}_\phi \in \accentset{\circ}{\mathfrak{B}}_{2}^m(\Omega)$ that are $\omega$-independent.}, then uniform convergence of $\mathcal{T}_{\phi,k}$ on compact subsets of $D_1$ is equivalent to convergence of the coefficient function $\phi_{2,k}$ in the Euclidean norm. Suppose that $\accentset{\circ}{\mathfrak{B}}_{2}^m$ is equipped with the topology of uniform convergence on compact subsets of $D_1$ through
$$d''(\mathcal{T}_{\phi}, \mathcal{S}_{\chi})=\sum_{n=2}^\infty \frac{1}{2^n}\sup_{z\in D_{1-\frac{1}{n}}}|T(z)-S(z)|.$$
If we assume $(\mathcal{T}_{\phi,k})_{k=1}^\infty \subset \accentset{\circ}{\mathfrak{B}}_{2}^m$ converges in $d^\prime$ to $\mathcal{T}_{\phi}\in \accentset{\circ}{\mathfrak{B}}_{2}^m$ then
\begin{align*}
    d''(\mathcal{T}_{\phi,k},\mathcal{T}_\phi)&=\sum_{n=2}^\infty \frac{1}{2^n}\sup_{z\in D_{1-\frac{1}{n}}}|T_k(z)-T(z)|\\
    &\leq \max_{z\in D_1} |T_k(z)-T(z))| \sum_{n=2}^\infty \frac{1}{2^n}\\
    &= d'(\mathcal{T}_{\phi,k},\mathcal{T}_\phi).
\end{align*}
Hence, convergence in $d'$ implies convergence in $d''$. Thus, if  $\mathcal{T}_{\phi,k}\to \mathcal{T}_\phi$ in $d'$, then $\phi_{2,k}\to \phi_2$ in the standard Euclidean norm, meaning that $\mathcal{T}_{\phi,k}\to \mathcal{T}_\phi $ in $d$ since $\Phi$ is a diffeomorphism. Further, if $\mathcal{T}_{\phi,k}\to \mathcal{T}_\phi$ in $d$, then
\begin{align}
    d'(\mathcal{T}_{\phi,k},\mathcal{T}_\phi)&=\max_{z\in S^1}\left| z\left(\frac{z-\phi_{2,k}}{1-\bar{\phi}_{2,k}z}- \frac{z-\phi_{2}}{1-\bar{\phi}_{2}z} \right) \right|\nonumber \\
     &\leq\max_{z\in S^1}\frac{(|z|^2+1)|\phi_{2,k}-\phi_{2}|+|z||\bar{\phi}_{2 ,k}\phi_{2}-\phi_{2,k}\bar{\phi}_{2}|}{|1-|z||\phi_{2,k}|||1-|z||\phi_{2}||} \nonumber\\
    &\leq C_k\left(2d(\mathcal{T}_{\Phi^{-1}_2(\phi),k},\mathcal{T}_{\Phi^{-1}_2(\phi)}) + d(\mathcal{T}_{(0,\Phi^{-1}(\mathfrak{Im}(\bar{\phi}_{2,k}{\phi}_{2})))},\mathcal{T}_{(0,0)}) \right)
    \label{eqn:Ck}
\end{align}
where $C_k = \frac{1}{|1-|\phi_{2,k}|||1-|\phi_2||}$ is uniformly bounded since $|\phi_{2,k}|,|\phi_2|<1$ and $\phi_{2,k}\to \phi_2$. This implies that $\mathcal{T}_{\phi,k}\to \mathcal{T}_\phi $ in $d'$ since $d(\mathcal{T}_{\Phi^{-1}_2(\phi),k},\mathcal{T}_{\Phi^{-1}_2(\phi)})$ and $ d(\mathcal{T}_{(0,\Phi^{-1}(\mathfrak{Im}(\bar{\phi}_{2,k}{\phi}_{2})))},\mathcal{T}_{0})$ tend to zero under the assumption that $\mathcal{T}_{\phi,k}\to \mathcal{T}_\phi$ in $d$. Finally, if $\mathcal{T}_{\phi,k}\to \mathcal{T}_\phi$ in $d''$, then we know that $\mathcal{T}_{\phi,k}\to \mathcal{T}_\phi$ in $d$, as discussed above. Since  $\mathcal{T}_{\phi,k}\to \mathcal{T}_\phi$ in $d$ if and only if  $\mathcal{T}_{\phi,k}\to \mathcal{T}_\phi$ in $d'$ we arrive at the following chain of implications for sequences in $\mathfrak{B}_{2}^m$,
$$\mathrm{Convergence \ in} \ d \iff \mathrm{Convergence \ in} \ d' \iff \mathrm{Convergence \ in}\ d''.$$
Unfortunately, these same relations do not stand for $(\mathcal{T}_{\phi,k})_{k=1}^\infty \subset \accentset{\circ}{\mathfrak{B}}_{2}^m(\Omega)$. In fact, convergence in $d$ gives rise to convergence in $d'$ only when stronger assumptions are made on the regularity of the coefficient function parametrising $\mathcal{T}_{\phi,k}$, and the limiting cocycle $\mathcal{T}_{\phi}$. If $|\phi_{2,\omega,k}|$ and $|\phi_{2,\omega}|$ are bounded away from $1$ uniformly over $\omega\in  \Omega$ and $k\in \mathbb N$, then one can uniformly control $C_k$ in \eqref{eqn:Ck} and relate
$\mathrm{convergence \ in} \ d$ with $\mathrm{convergence \ in}\ d'.$
\end{remark}

\begin{remark}
    {Although it does not take a crucial part in the majority of the proofs that follow, \prop{prop:cmtvs} is essential to apply the prevalence framework to $\accentset{\circ}{\mathfrak{B}}_{2}^m(\Omega)$. Alluding to \dfnn{def:Prevalent} we require this additional vector space structure on $\accentset{\circ}{\mathfrak{B}}_{2}^m(\Omega)$ which we have now obtained. }
\end{remark}
\noindent
Our main result for this section illustrates that, when $\Omega=S^1$, stable monic quadratic Blaschke product cocycles fixing the origin are prevalent among cocycles parametrised by $C^1$ coefficient functions. Before proving the main result of this section, we recall the following result.

\begin{lemma}[Sard's Theorem]
    {Suppose $M$ and $N$ are $C^1$ manifolds with or without boundary, and $F : M \to N$ is a $C^1$ map. If $\dim (M) < \dim (N)$ , then $F(M)$ has Lebesgue measure zero in $N$.}
    \label{lem:diffcurve_leb0}
\end{lemma}
{\begin{example}
    Suppose $M=S^1$, $N=\mathbb{R}^n$ and $F\in C^1(S^1,\mathbb{R}^n)$. Then, Sard's theorem tells us that $\leb({F(S^1)})=0$ in $\mathbb{R}^n$.
\end{example}}
\label{sec:prev_proof}
\begin{theorem}
    Let $\Omega = S^1$ and suppose that { $\mathbb P$ has full support}. Then stability is prevalent in
    $$\mathcal{E} = \left\{ {\mathcal{T}}_{\zeta} \in \accentset{\circ}{\mathfrak{B}}_{2}^{m}(\Omega) \ \big| \ \zeta \in C^1(\Omega,\{0\}\times D_{1}) \right\}.$$
    \label{thrm:prev_thrm}
    \end{theorem}
    \begin{proof}
     Let $v(\lambda)=(0,\lambda)\in D_1^2$ and set $\mathcal{P}=\{{\mathcal{T}}_{v(\lambda)} \ | \ \lambda \in D_1\}$. We claim that $\mathcal{P}$ is a probe for $\accentset{\circ}{\mathfrak{B}}_{2}^{m,S}(\Omega)$.
     We wish to show that the set of all
     $\lambda \in D_1$ for which $\mathcal{T}_{\zeta}+{\mathcal{T}}_{v(\lambda)}$ is unstable, has zero Lebesgue measure. Namely, for quadratic cocycles, through {Corollary~\ref{cor:instab_cond},}  we aim to show that
    $$\leb\left(\left\{\lambda\in D_1 \ \big| \ \min_{\omega\in\Omega} \left| \zeta_{2,\omega}^\lambda\right|=0\right\}\right)=0.$$
    Since $\Phi$ is a diffeomorphism
    \begin{align*}
   & \left\{\lambda\in D_1 \ \big| \ \min_{\omega\in\Omega} \left| \zeta_{2,\omega}^\lambda\right|=0\right\}\\&=\left\{\lambda\in D_1 \ \big| \ \Phi^{-1}(\Phi(\zeta_{2,\omega})+\Phi(\lambda))=0\text{ for some }\omega \in \Omega \right\} \\
    &=\left\{\lambda\in D_1 \ \big| \ \lambda = -\zeta_{2,\omega} \text{ for some }\omega \in \Omega\right\}\\
    &= \left\{ -\zeta_{2,\omega} \ | \ \omega \in \Omega \right\}.
    \end{align*}
    Finally, by \lem{lem:diffcurve_leb0}, because $\zeta_{2}\in C^1(\Omega,D_{1})$ we have $\leb\left( \left\{ -\zeta_{2,\omega} \ | \ \omega \in \Omega \right\} \right)=0$ and $\mathcal{P}$ is a probe.
    \end{proof}
\begin{remark}
    We require $C^1$ regularity of coefficient functions, and not simply continuity, due to the existence of space filling curves \cite{zero_measure_c1curve}.
\end{remark}
\section{Prevalence for higher dimensional probability spaces}
\label{sec:nonprev}
{Thus far we have considered cocycles over the circle. A natural question to ask is: {what happens when a higher dimensional $\Omega$ is considered?} We now address this query under the assumptions of \thrm{thrm:prev_thrm}, but taking $\Omega = B_\delta(0)\subset \mathbb{R}^n$.
We prove a negative result for $\dim (\Omega) = 2$. 
A similar result holds when
$\Omega = B_\delta(0)\subset \mathbb{R}^n$ ($n\geq 2$). The only difference there is that one cannot assume that $\partial_\omega \zeta_{2,\omega}$ is invertible as in the proof of \thrm{thrm:noprobe}. This issue can be avoided by considering the restriction to $B_\delta(0)\cap (\mathbb{R}^2 \times \{(0, \dots, 0)\})$ and following the two-dimensional argument.
\begin{theorem}\label{thrm:noprobe}
    Fix {$0<\delta<1$}, let $\Omega=B_\delta(0)\subset \mathbb{R}^2$  and suppose that { $\mathbb P$ has full support}. If
    $$\mathcal{E} = \left\{ {\mathcal{T}}_{\zeta} \in \accentset{\circ}{\mathfrak{B}}_{2}^{m}(\Omega) \ \big| \ \zeta \in C^1(\Omega,\{0\}\times D_{1}) \right\},$$
    then there does not exist a probe for $\mathcal{E}^S = \mathcal{E} \cap {\mathfrak{B}}_{2}^{m,S}(\Omega)$.
    \end{theorem}
    \begin{proof}
        For each $i=1,\dots,d$ let $\varphi_i\in C^1(\Omega, D_1)$ and set $v_{i,\omega}=(0,\varphi_{i,\omega})\in D_1^2$. Take any $d$-dimensional subspace of $\mathcal{E}$ given by $\mathcal{P}=\{ \sum_{i=1}^d \mathcal{T}_{\lambda_i v_{i,\omega}} \ | \ \lambda_i \in \mathbb{C} \ \mathrm{for \ each} \ i=1,\dots,d\}$. We claim that this cannot be a probe for $\mathcal{E}^S$. Indeed,  {relying on Corollary~\ref{cor:instab_cond}} and recalling \dfnn{def:addition_of_fbp}, we aim to show that {there exists $\zeta_2\in C^1(\Omega, D_1)$ such that} the set of solutions $\lambda = (\lambda_1,\cdots,\lambda_d)$ to
        \begin{equation}
            \min_{\omega \in \Omega}\left|\Phi^{-1}\left(\Phi(\zeta_{2,\omega})+\sum_{i=1}^d\Phi( \lambda_i\varphi_{i,\omega})\right)\right|=0
            \label{eqn:not_probe}
        \end{equation}
        has positive $2d$-dimensional Lebesgue measure. Define
        $$f(\lambda,\omega) := \Phi(\zeta_{2,\omega})+\sum_{i=1}^d\Phi( \lambda_i\varphi_{i,\omega}).$$
        Since $\Phi$ is a diffeomorphism, $\Phi(0)=0$  and $\zeta_2,\varphi_i$ are continuously differentiable for each $i=1,\dots, d$, solutions to \eqref{eqn:not_probe} are identical to the set $\lambda$ satisfying $f(\lambda,\omega)=0$  {for some $\omega\in\Omega$}. Let $\zeta_2\in C^1(\Omega, D_1)$ be such that  {there exists $\omega_0\in\Omega$  where} $\zeta_{2,\omega_0}=0$ and $\partial_\omega \zeta_{2,\omega_0}$ is invertible\footnote{{For example, one could let $\omega_0=\textbf{0} \in \mathbb{R}^2$ and $\zeta_{2,\omega}=\omega$. Then $\zeta_{0,\omega_0}=0$ and $\partial_\omega\zeta_{2,\omega_0}=Id$.}}. Then $f(0,\omega_0)=0$ and by the chain rule
        \begin{align*}
            \partial_\omega f(0,\omega_0)&=D\Phi(\zeta_{2,\omega_0})\partial_\omega \zeta_{2,\omega_0}\\
            &=\partial_\omega \zeta_{2,\omega_0}
        \end{align*}
        is invertible, where $D\Phi(0)=\mathrm{Id}$ by \eqref{eqn:DPhi}. Therefore, by the Implicit Function Theorem, there exists an open set $U\subset D_1^d$  containing zero and a function $g\in C^1(U,\Omega)$ such that $f(\lambda,g(\lambda))=0$ (here $D_1^d$ denotes the $d$-dimensional unit disk). Since \eqref{eqn:not_probe} is satisfied for all $\lambda \in U$, $\mathcal{P}$ is not a probe.
    \end{proof}
\begin{remark}
    We point out that although the existence of a probe is only a sufficient condition for a set to be prevalent, \thrm{thrm:noprobe} provides us with strong evidence that stability is not prevalent in $\mathcal{E}$ when $\Omega = B_\delta(0)\subset \mathbb{R}^n$.
\end{remark}
\begin{remark}
    {In the space of monic quadratic Blaschke product cocycles fixing the origin, the stability criteria in \prop{prop:stab_instab} is reduced to studying the finite-dimensional mapping $\omega\mapsto \zeta_{2,\omega}$ where $\zeta_{2}\in C^1(\Omega,D_1)$. The results of the last two sections may be roughly summarised as follows. In the case that $\dim(\Omega)=1$, the range of $\omega\mapsto \zeta_{2,\omega}$ generically avoids the origin under perturbations. However, when $\dim(\Omega)\geq 2$, the intersection of the image of $\omega\mapsto \zeta_{2,\omega}$ and the origin is non-empty under a set of perturbations with positive Lebesgue measure. In this sense, when $\dim(\Omega) = 1$ one can regard \textit{instability of Lyapunov spectrum} as a codimension 1 phenomenon, generically only occuring along a discrete set of points in a one-parameter family of cocycles. When
    $\dim(\Omega) >1$,  instabilities of the Lyapunov spectrum can occur robustly.}
\end{remark}
\begin{remark}
    It is plausible that the arguments made in the proofs of \thrm{thrm:prev_thrm} and \thrm{thrm:noprobe} could be used to show that a probe does not exist in the setting of \thrm{thrm:prev_thrm} with the relaxed assumption that coefficient functions are continuous instead of continuously differentiable,
    as the image of a continuous function (for example a space filling curve) could accumulate and contain a full neighbourhood about the origin.
\end{remark}

    }

\section{Stability for almost every Blaschke product cocycle fixing the origin}\label{sec:aefbcstab}
Thus far we have considered prevalence of stability for monic quadratic Blaschke product cocycles that fix the origin. In this section, we divert our focus towards studying general cocycles that fix the origin. Here the theory of prevalence does not apply since $\accentset{\circ}{\mathfrak B}_n(\Omega)$ lacks a natural linear structure. We instead introduce a one parameter complex family of perturbations for Blaschke product cocycles which play the role of a probe. This allows us to show that
almost every perturbation of any fixed smooth Blaschke product cocycle fixing the origin is stable.
Section~\ref{ss:degreen} addresses the case of cocycles of fixed degree. In  Section~\ref{ss:genDeg}, we generalise our results to the case where the degree is $\omega$-dependent.
\\ \\ \noindent
Although monic Blaschke product cocycles are determined by their zeros $\zeta_{i,\omega} \in D_1^n$, the mapping $\zeta \mapsto \mathcal{T}_{\zeta}$ is not one-to-one for $n>1$. This follows from the fact that the resulting cocycle is invariant under permutations of the coordinates of $\zeta$. In fact, the space of Blaschke products of degree $n$ satisfies $\mathfrak{B}_n\cong S^1 \times D_1^n / \sim$, where $(\rho,\zeta)\sim (\rho', \zeta')$ if and only if $\rho=\rho'$ and there exists a coordinate permutation $\pi$ such that $\zeta=\pi(\zeta')$. This equivalence relation is discussed in further detail in \cite{PropertiesaceFBP}. This observation means that the linear structure defined in \dfnn{def:addition_of_fbp} cannot be extended to cocycles with degree $n>2$. This is illustrated through the following example.
\begin{example}
    Take $\mathcal{T}_{\zeta},\mathcal{T}_{\zeta^\prime}\in \accentset{\circ}{\mathfrak{B}}_3^m(\Omega)$ where $\zeta = (0,\zeta_{2},\zeta_{3})$ and $\zeta^\prime = (0,\zeta_{3},\zeta_{2})$, with $\zeta_2\neq\zeta_3$. Recall that $\mathcal{T}_{\zeta}\in \accentset{\circ}{\mathfrak{B}}_3^m(\Omega)$ satisfies $\rho(\omega)=1$ for every $\omega\in\Omega$ and thus $(\rho,\zeta) \sim (\rho^\prime,\zeta^\prime)$. If one were to naturally extend \dfnn{def:addition_of_fbp} to monic cubic Blaschke product cocycles fixing the origin by defining $ \mathcal{T}_{\xi}+\mathcal{T}_{\chi}:={\mathcal{T}}_{\Phi_3^{-1}(\Phi_3(\xi)+\Phi_3(\chi))}$ for $\mathcal{T}_{\xi},\mathcal{T}_{\chi}\in \accentset{\circ}{\mathfrak{B}}_3^m(\Omega)$, then  although $(\rho,\zeta) \sim (\rho^\prime,\zeta^\prime)$, $\mathcal{T}_{\zeta}+\mathcal{T}_{\zeta^\prime}$ is not equivalent to $\mathcal{T}_{\zeta}+\mathcal{T}_{\zeta}$ since $\Phi_3^{-1}(\Phi_3(\zeta)+\Phi_3(\zeta^\prime))\neq \Phi_3^{-1}(\Phi_3(\zeta)+\Phi_3(\zeta))$.
\end{example}

\subsection{Stability for almost every smooth cocycle of fixed degree}
\label{ss:degreen}
As opposed to addressing stability of cocycles in $\accentset{\circ}{\mathfrak{B}}_n(\Omega)$ from a prevalence standpoint, we instead introduce a perturbative method similar to that of a probe which allows us to obtain a stability result for ``almost every'' element of $\accentset{\circ}{\mathfrak{B}}_n(\Omega)$.
\begin{proposition}\label{prop:pert}
Let ${\mathcal{T}}_{(\rho,\zeta)}\in\accentset{\circ}{\mathfrak{B}}_{n}(\Omega)$. For each $\lambda\in D_1$, let ${\mathcal{T}}_{(\rho,\zeta)}^\lambda$
be the cocycle parametrised by $\rho$ and
\begin{equation}\label{eq:zlambda}
\zeta_{i,\omega}^\lambda:=\begin{cases}0, \ \ \ \ \ &i=1 \\
\Phi^{-1}(\Phi(\zeta_{i,\omega})+\Phi(\lambda)), &i=2,\dots,n.
\end{cases}
\end{equation}
Then, ${\mathcal{T}}_{(\rho,\zeta)}^\lambda \in\accentset{\circ}{\mathfrak{B}}_{n}(\Omega)$.
\end{proposition}
\begin{proof}
For every $\lambda\in D_1$, $\zeta_{1,\omega}=0$ and for every $1\leq i\leq n$, $\zeta_{i,\omega}^\lambda\in D_1$. Thus, ${\mathcal{T}}_{(\rho,\zeta)}^\lambda \in\accentset{\circ}{\mathfrak{B}}_{n}(\Omega)$. To conclude, it is enough to show that the cocycle $\mathcal{T}_{(\rho,\zeta)}^\lambda$ is well defined. That is, that   $\mathcal{T}_{(\rho,\zeta)}^\lambda$ is invariant under permutations of the coordinates of $\zeta$. This follows from the fact that if
$(\zeta_{1}, \dots, \zeta_{n})\sim (\zeta'_{1}, \dots, \zeta'_{n})$ then $(\zeta_{1}^\lambda, \dots, \zeta_{n}^\lambda)\sim (\zeta_{1}^{'  \lambda}, \dots, \zeta_{n}^{' \lambda})$ for every $\lambda\in D_1$.
\end{proof}
\begin{remark}
Note that $\lambda \in D_1$ in fact describes a family of perturbations in $\mathbb{C}$ through \eqref{eq:zlambda} since $\Phi(\lambda)\in\mathbb{C}$.
\end{remark}
\begin{theorem}\label{thm:aePertDegn}
    Let $\Omega = S^1$ and suppose that ${\mathcal{T}}_{(\rho,\zeta)}\in\accentset{\circ}{\mathfrak{B}}_{n}(\Omega)$, where $\zeta:\Omega \to \{0\}\times D_1^{n-1}$ is continuously differentiable, {and $\mathbb P$ has full support}. Then for almost every $\lambda\in D_1$, the cocycle
    ${\mathcal{T}}_{(\rho,\zeta)}^\lambda$ is stable. That is,
    $$\leb \left(\left\{ \lambda\in D_1 \ \big| \ \mathcal{T}_{(\rho,\zeta)}^\lambda \in \accentset{\circ}{\mathfrak{B}}_{n}^{U}(\Omega) \right\}\right)=0 .$$
    \end{theorem}
\noindent
 Theorem~\ref{thm:aePertDegn} is a special case of Theorem~\ref{thm:aePert}, which will be proved in the next section.

\subsection{Stability for almost every smooth cocycle of varying degree}
\label{ss:genDeg}
The perturbations introduced in \prop{prop:pert} can be generalised and applied to the space of random Blaschke products that fix the origin, $\accentset{\circ}{\mathfrak{B}}(\Omega)$, as in \dfnn{def:fbp_fixorig}. Such perturbations can be defined as follows. For $\lambda\in D_1$, and ${\mathcal{T}}_{(n,\rho,\zeta)}\in \accentset{\circ}{\mathfrak B}(\Omega)$,
let
    $${\mathcal{T}}_{(n,\rho,\zeta)}^\lambda(\omega, z)= \rho_\omega z \prod_{i=2}^{n_\omega}\frac{z-\zeta^\lambda_{i,\omega}}{1-\bar{\zeta}^\lambda_{i,\omega}z},$$
    be the perturbed cocycle where $\zeta^\lambda_{i,\omega}$ is as in \eqref{eq:zlambda}. We note that the $\omega$ dependence of $n$ does not influence the results of \prop{prop:est rT} and \prop{prop:admis_cond}, so \cor{cor:BaeqB} tells us that all elements of $\accentset{\circ}{\mathfrak B}(\Omega)$ are admissible and hence \thrm{the:Stability of Lyapunov Spectrum} applies to elements of $\accentset{\circ}{\mathfrak B}(\Omega)$. Further, the degree of the cocycle in \prop{prop:stab_instab} is not essential. Thus, the stability and instability conditions for elements of $\accentset{\circ}{\mathfrak{B}}(\Omega)$ can be determined though those of $\accentset{\circ}{\mathfrak{B}}_n(\Omega)$ by replacing $n$ with $n_\omega$.
    \begin{proposition}[Stability conditions for ${\mathcal{T}}_{(n,\rho,\zeta)}\in \accentset{\circ}{\mathfrak B}(\Omega)$]
Suppose that ${\mathcal{T}}_{(n,\rho,\zeta)} \in \accentset{\circ}{\mathfrak{B}}_{}(\Omega)$. Then $\mathcal{T}_{(n,\rho,\zeta)}$ is:
\begin{enumerate}
    \item Stable if
    $$\essinf_{\omega\in\Omega}\left|\prod_{i= 2}^{n_\omega}\zeta_{i,
    \omega}\right|>0.$$
    \item Unstable if
    $$\essinf_{\omega\in\Omega}\left|\prod_{i= 2}^{n_\omega}\zeta_{i,
    \omega}\right|=0.$$
\end{enumerate}
\label{prop:stab_instab2}
\end{proposition}\noindent
Those cocycles in $\accentset{\circ}{\mathfrak B}(\Omega)$ satisfying the stability and instability condition in \prop{prop:stab_instab2} belong to the spaces
$\accentset{\circ}{\mathfrak{B}}_{}^{S}(\Omega)$ and $\accentset{\circ}{\mathfrak{B}}_{}^{U}(\Omega)$, respectively.

\begin{theorem}\label{thm:aePert}
    Let $\Omega = S^1$ and suppose {that $\mathbb P$ has full support}. For each $m\geq2$, let the set $\Omega_m = \{\omega\in \Omega \ | \ n_\omega =m\}$ be a (possibly empty) collection of intervals. Suppose that ${\mathcal{T}}_{(n,\rho,\zeta)}\in\accentset{\circ}{\mathfrak{B}}(\Omega)$ is such that for each $m\geq 2$ for which $\Omega_m\neq \emptyset$, the restriction $\zeta|_{\Omega_m}: \Omega_m \to \{0\}\times D_1^{m-1}$ can be extended to a continuously differentiable function
    $\zeta_m: \overline{\Omega}_m \to \{0\}\times D_1^{m-1}$. Then, for almost every $\lambda\in D_1$, the cocycle
    ${\mathcal{T}}_{(n,\rho,\zeta)}^\lambda$ is stable. That is,
    $$\leb \left(\left\{ \lambda\in D_1 \ \big| \ \mathcal{T}_{(n,\rho,\zeta)}^\lambda \in \accentset{\circ}{\mathfrak{B}}^{U}(\Omega) \right\}\right)=0 .$$

    \begin{proof}
Using Proposition~\ref{prop:stab_instab2}, we aim to show that
    $$\leb\left(\left\{\lambda\in D_1 \ \big| \ \essinf_{\omega\in\Omega} \left|\prod_{i=2}^{n_\omega} \zeta_{i,\omega}^\lambda\right|=0\right\}\right)=0.$$
    For each $\omega \in \overline{\Omega}_m$, let $\zeta_{m. i,\omega}^\lambda$ be defined as in
    \eqref{eq:zlambda} with $\zeta$ replaced by $\zeta_m$. Since $\Phi$ is a diffeomorphism satisfying $\Phi(-z)=-\Phi(z)$ and {$\mathbb P$ has full support}, then
    \begin{align*}
        \left\{\lambda\in D_1 \ \big| \ \essinf_{\omega\in\Omega} \left|\prod_{i=2}^{n_\omega} \zeta_{i,\omega}^\lambda\right|=0\right\}
        &=\bigcup_{m=2}^\infty \Bigg\{\lambda\in D_1 \ \big| \ \prod_{i=2}^{m} \zeta_{m, i,\omega}^\lambda=0  \\
        &\quad \  \text{for some }\omega \in \overline{\Omega}_m\Bigg\}\\
        &=\bigcup_{m=2}^\infty\big\{\lambda\in D_1 \ \big| \ \zeta_{m, i,\omega}^\lambda=0  \ \mathrm{for \ some} \ \omega \in \overline{\Omega}_m,\\ &\quad \  i=2,\dots,m \big\}\\
         &=\bigcup_{m=2}^\infty\big\{\lambda\in D_1 \ \big| \ \lambda=-\zeta_{m, i,\omega} \ \mathrm{for \ some} \ \omega \in \overline{\Omega}_m,\\ &\quad \  i=2,\dots,m \big\}.
    \end{align*}
    Thus
    \begin{align*}
        \leb \left(\left\{\lambda\in D_1 \ \big| \ \essinf_{\omega\in\Omega} \left|\prod_{i=2}^{n_\omega} \zeta_{i,\omega}^\lambda\right|=0\right\}\right)&=\leb\left(\bigcup_{m=2}^\infty \bigcup_{i=2}^{m} \left\{  - \zeta_{m,i,\omega} \ | \ \omega\in\overline{\Omega}_m \right \}\right)\\
        &\leq \sum_{m=2}^\infty\sum_{i=2}^m \leb(\left\{  - \zeta_{m,i,\omega} \ | \ \omega\in\overline{\Omega}_m \right \})\\
        &=0,
    \end{align*}
    because $\zeta_{m,i}\in C^1(\overline{\Omega}_m,D_{1})$ for each $i=2,\dots,m$.
    \end{proof}
    \end{theorem}
\section*{Acknowledgments}
We thank the Mathematical Research Institute (MATRIX), where part of this research was performed. The authors acknowledge support from the Australian Research Council (DP220102216). JP acknowledges the Australian Government Research Training Program for financial support. We thank Zoe H for assisting us in generating the TikZ picture in \fig{fig:maptoC}, and the anonymous referees for their valuable comments and suggestions.
\footnotesize
\bibliographystyle{plain}
\bibliography{prevalence}

\end{document}